\documentclass[sn-mathphys]{sn-jnl}

\theoremstyle{thmstyleone}%
\newtheorem{theorem}{Theorem}
\newtheorem{proposition}[theorem]{Proposition}%

\theoremstyle{thmstyletwo}%
\newtheorem{example}{Example}%
\newtheorem{remark}{Remark}%
\theoremstyle{thmstylethree}%
\newtheorem{lemma}{Lemma}[section]
\newcommand{\E}{\mathbb{E}}
\newcommand{\PP}{\mathbb{P}}
\newcommand{\RR}{\mathbb{R}}

\def\br{\breve}
 
\def\e{\varepsilon}   
    
   \def\s{\sigma}
    
\def\tr{\triangle}

\def\<{\langle} \def\>{\rangle}

\def\nn{\nonumber}

\def\bc{\begin{center}}       \def\ec{\end{center}}
\def\ba{\begin{array}}        \def\ea{\end{array}}
\def\be{\begin{equation}}     \def\ee{\end{equation}}
\def\bea{\begin{eqnarray}}    \def\eea{\end{eqnarray}}
\def\beaa{\begin{eqnarray*}}  \def\eeaa{\end{eqnarray*}}
\def\la{\label}
\def\v{\vert}

\begin{document}

\title[Explicit approximation of the invariant measure for  SDDEs]{Explicit approximation of the invariant measure for SDDEs with the nonlinear diffusion term}


\author[1]{\fnm{Xiaoyue} \sur{Li}}\email{lixy209@nenu.edu.cn}
\equalcont{These authors contributed equally to this work.}
\author[2]{\fnm{Xuerong} \sur{Mao}}\email{x.mao@strath.ac.uk}
\equalcont{These authors contributed equally to this work.}

\author*[1]{\fnm{Guoting} \sur{Song}}\email{songgt606@nenu.edu.cn}

\affil*[1]{\orgdiv{School of Mathematics and Statistics}, \orgname{Northeast Normal University}, \orgaddress{\street{No. 5268, Renmin Street}, \city{Changchun}, \postcode{130024}, \state{Jilin}, \country{China}}}

\affil[2]{\orgdiv{Department of Mathematics and Statisticst}, \orgname{University of Strathclyde}, \orgaddress{\street{26 Richmond Street}, \city{Glasgow}, \postcode{G1 1XH}, 

\country{UK}}}



\abstract{To our knowledge,  the existing measure approximation theory  
requires the diffusion term of  the stochastic  delay differential equations  (SDDEs) to be globally Lipschitz continuous.
Our work is to develop a new explicit numerical method for SDDEs with the nonlinear diffusion term  and establish the measure approximation theory. Precisely,
 we construct a function-valued  explicit truncated Euler-Maruyama segment process (TEMSP) and prove that it admits a unique ergodic numerical invariant measure. We also prove that the numerical invariant measure converges to the underlying one of SDDE in the Fortet-Mourier distance.  Finally, we give an example and numerical simulations to support our theory.}

\keywords{Stochastic delay differential equations, truncated Euler-Maruyama segment proces, stability in distribution, numerical invariant measure}



\maketitle

\section{Introduction}\label{sec1}

\label{Tntr}

Since time delays are omnipresent and entrenched in real systems, delay differential equations
have a wide range of emerging and existing
applications in, for instance, physics, biology, medical sciences, automatic control systems; see, e.g., \cite{Kolmanovskii-Nosov, Monk2003, Nicaise-Pignotti,Shaikhet}.
On the other hands, systems in the real word are always subject to environmental noise. Stochastic  delay  differential equations (SDDEs) have become more and more popular mathematical models for many real systems; see e.g. \cite{Mao2007, 
Mcshame, Mohammed} and references therein.  
Asymptotic stability is one of the most important topics in the study of SDDEs.
There are two fundamental categories: (ASE) asymptotic stability of an equilibrium state; (ASD) asymptotic stability in distribution.
ASE is to study whether the solutions of a given SDDE system will tend to the equilibrium state (e.g., 0 as in most papers) in moment or in probability; while ASD is to  study whether the probability distributions of the solutions of the given SDDE system will converge to a probability distribution, known as  stationary distribution or invariant measure.
There is an intensive literature on ASE
(see, e.g.,  \cite{Khasminskii,Kolmanovskii-Nosov,Mao2007,Mohammed} and many others). The literature on ASD is  much less than ASE but has been growing quickly for the past 10 years  (see, e.g.,  \cite{bao-yin-yuan, Basak-Bhattacharya, Butkovsky, Cloez-Hairer,Kulik-Scheutzow, Xie-Zhang}).    The reason why there are  fewer papers on ASD than ASE is because the mathematics involved is much more complicated than that used for the study of ASE but certainly not because ASD is less important.
In fact,  it is inappropriate to study ASE for many SDDE systems in the real world  but more appropriate to study ASD.  For example,
for many population systems under random environment,
 the stochastic permanence
 is a more desired control objective than the extinction (see, e.g., \cite{BM04a,BM04b,HN18}).
In this situation it is useful to investigate whether or
not the probability distribution of the solutions will converge to
a probability distribution (i.e., ASD),  but not to zero (i.e.,
ASE) (see, e.g., \cite{HN18,M11,TDHS06}).
The two stability categories can also be illustrated by the control of Covid-19.
 There are essentially two
control strategies: one is to suppress infected to 0 but the
other is to live with Covid-19. The former is to stabilise
the infected to 0 with probability 1 (i.e.,  ASE), while the latter is
to stabilise the distribution of the infected to a stationary
distribution (i.e., ASD). More details on ASD and the related ergodic theory can be found in, e.g., \cite{Hairer2005, Hairer2008, Hairer-Mattingly-Scheutzow, hairer-ohashi, WangYa,mao-yuan2003}.

Although there is a theory on the existence and uniqueness of the invariant measure of the segment solution process to an SDDE, there is so far no way to obtain the theoretical cumulative probability distribution (CPD) of the invariant measure.  It is therefore significant to establish numerical methods to approximate the CPD.
For the past ten years,  several numerical schemes have been proposed  to approximate the CPD of the invariant measure of a stochastic differential equation (SDE) (see, e.g.,  \cite{Fang-Giles, Lamba-Mattingly-Sstuart, Li-Ma-Yang-Yuan, Li2018, Mattingly-Staurt-Higham, Mei-Yin2015, Pages-Panloup2012, ZhangX} and references therein), where the invariant measure is distributed on the finite-dimensional space $\RR^d$.  
However, the  invariant measure of  an SDDE is distributed on the infinite-dimensional space $C([-\tau,0;\RR^d)$.  It is hence much harder to approximate its CPD numerically. 
Nevertheless,  some progress has been made recently in this direction under the condition that the diffusion coefficient of the underlying infinite-dimensional system is globally Lipschitz continuous. We would like to mention \cite{C-H-S,H-W,H-W-Z}  for the study of stochastic partial differential equations and  \cite{bao-shao-yuan, NNN} for  SDDEs.
On the other hand, the global Lipschitz condition is very restrictive and most of SDDE models in applications do not satisfy it (see, e.g., \cite{BM04a,BM04b,DLMP08,HM10}).  It therefore becomes necessary and urgent  to design numerical methods to approximate the invariant measure of  an SDDE whose   coefficients are only locally Lipschitz continuous.  From the point of computational cost, it is more desired that the numerical methods are explicit.  This is our main aim in this paper.

Consider an SDDE described by
\begin{align}\la{n1.1}
\mathrm{d}x(t)=f(x(t),x(t-\tau))\mathrm{d}t +g(x(t),x(t-\tau))\mathrm{d}W(t),~~~t> 0
\end{align}
with an initial data $x(\theta)=\xi(\theta)$, $\theta\in [-\tau, 0]$,
where  $\tau>0$,  $f:  \RR^d \times \RR^d \rightarrow \RR^{d}$ and $g: \RR^d\times \RR^d  \rightarrow \RR^{d\times m}$ are Borel measurable, $W(t)$ is an m-dimensional Brownian motion on a probability space  $( \Omega,~\mathcal{F}, ~\PP )$ with a right-continuous complete filtration $\{{\mathcal{F}}_{t}\}_{t\geq 0}$, and $\xi$ is an $\mathcal F_0$-measurable, continuous function-valued random variable from
$[-\tau,~0]$ to $\mathbb{R}^{d}$.
Let $\{x_{t}\}_{t\geq 0} $ be the  segment process, where $x_{t}(\theta):=x(t+\theta)$ for $\theta\in[-\tau,0]$.
As we know, approximation of the invariant measure in the infinite horizon is quite  challenging for nonlinear SDDEs with non-globally Lipschitz  diffusion coefficient. We mainly  face the following difficulties:
\begin{itemize}
\item
Mathematically speaking, the time-homogenous Markov property  of  the  segment process of SDDEs plays a crucial role in investigating the ergodicity.
In the numerical case, how to construct a  continuous function-valued  explicit numerical  time-homogenous Markov process  $ \{Y_{t_k}^{\xi,\tr}\}_{k\geq 0}$ for approximating
the invariant measure?
\item Generally, the tightness of  measures  is often used to derive the existence of an invariant measure.
In the infinite-dimensional space, $\sup_{k\geq0}\E\|Y_{t_k}^{\xi,\tr}\|^2<\infty$   fails to  imply the tightness of correspding measures  $\{\mu_{t_k}^{\xi,\triangle}\}_{k\geq 0} $ since the relative compactness does not follows from the boundedness. The existence of numerical invariant measures therefore needs to be proved carefully.
\item
The super-linear diffusion coefficient makes it difficult to obtain the attraction of second moment of $ \{Y_{t_k}^{\xi,\tr}\}_{k\geq 0}$, i.e.,
$$\E\|Y_{t_k}^{\xi,\tr}-Y_{t_k}^{\zeta,\tr}\|^2\leq \phi (t_k)\E\|\xi-\zeta\|^2,$$
where $\phi : [0, \infty)\rightarrow \RR_+$  satisfies
 $\lim_{t\rightarrow \infty}\phi(t)=0$.
Hence we need to explore  the attraction in probability or in distribution and further discuss the uniqueness of numerical invariant measures.
\end{itemize}

To investigate the measure approximation of  nonlinear SDDEs with non-globally Lipschitz  diffusion coefficient,  we  and our coauthors  in \cite{Song-Li2021} first constructed an  explicit  truncated EM  method and proved  numerical solutions strongly convergent to the exact ones in the finite horizon.
Furthermore, in current paper we improve  the truncated EM method  given by \cite{Song-Li2021} and borrowing the linear interpolation go a further step to design  the continuous function-valued  numerical time-homogenous Markov process  $ \{Y_{t_k}^{\xi,\tr}\}_{k\geq 0} $ (see \eqref{n3.5}--\eqref{temlis} for details), called TEMSP.  By the strong convergence of numerical solutions ($\RR^d$-valued)  in \cite{Song-Li2021} we obtain the weak convergence of
$\{Y_{t_k}^{\xi,\tr}\}_{k\geq 0}$  (continuous function-valued) in the finite horizon (see Lemma \ref{noth2} for details).
Moreover, taking advantage of TEMSP, we yield the uniform boundedness  and the attraction of $ \{Y_{t_k}^{\xi,\tr}\}_{k\geq 0} $  in probability (see Proposition \ref{leerry} and  Proposition \ref{lemma5} for details). Making use of above propositions we reveal that the sequence of numerical measures $\{\mu_{t_k}^{\xi,\triangle}\}_{k\geq 0} $ is Cauchy in the Fortet-Mourier distance $d_{\Xi}$  defined by \eqref{n2.1}.  Together with the completeness
of $(\mathcal P(C), d_{\Xi})$ (see Page 4) we prove the existence of numerical invariant measure. Furthermore, the uniqueness of numerical invariant measures follows from the  attraction of $ \{Y_{t_k}^{\xi,\tr}\}_{k\geq 0} $  in probability. Finally,  utilizing weak convergence of
$\{Y_{t_k}^{\xi,\tr}\}_{k\geq 0}$  we prove the numerical invariant measure converges to the underlying one  in $d_{\Xi}$ as the step size tends to zero.

The rest of this paper is arranged as follows.    Section \ref{dcsNP}  introduces some notations and cites some used results.
  Section \ref{dcsMR}  proposes  the truncated EM linear interpolation scheme and states the
main results.
   Section \ref{proofs} gives the  proofs in details.  Section \ref{exper} gives an example and numerical simulations  to illustrate  our results. Finally,  Section \ref{summ}  concludes this paper.
\section{Preliminaries}\label{dcsNP}
In the beginning of this section, we introduce some notations.
Denote by $\lvert\cdot\rvert$ the Euclidean norm in $\RR^d$ and the trace
norm  in $\RR^{d\times m}$, and by $\langle \cdot,\cdot\rangle$ the inner product in $\RR^{d}$.
For real numbers $a$ and $b$,  let $a\vee{}b=\max\{a,b\}$ and $a\wedge{}b=\min\{a,b\}$, respectively.
Let $\lfloor a\rfloor$ be the integer part of the real
number $a$. Let  $\boldsymbol{1}_A(x)$ be the indicator function of the set $A$.
 Let $\mathbb{R}_{+}=[0,\infty)$. Denote by { $C:=C([-\tau,~0];~\mathbb{R}^{d})$}  the  family of continuous functions $X$ from
$[-\tau,~0]$ to $\mathbb{R}^{d}$ with the supremum norm $\|X\|
=\sup_{-\tau\leq\theta\leq0}\lvert X(\theta)\rvert$. For $M>0$, define $B(M):=\{X\in C: \|X\|\leq M\}$ and $B^c(M)$ is its complementary in $C$. 
For $p\geq2$,  denote by $\mathcal C_{\mathcal F_0}^{p}:=\mathcal C_{\mathcal F_0}^{p}([-\tau,~0];~ \RR^d)$  the  family of $\mathcal F_0$-measurable, $C$-valued random variables
such that  $\sup_{-\tau\leq\theta\leq0}\E{\lvert \eta(\theta)\rvert}^p<\infty$.
 For $M>0$ and $p\geq2$, define
$\mathbb{B}(M,p):=\{\xi\in\mathcal C_{\mathcal F_0}^{p} : \E\|\xi\|\leq M\}$ and $\mathbb{B}^{c}(M,p)$ is its complementary in $\mathcal C_{\mathcal F_0}^{p}$. 
Denote by $ \hat{ C}(\RR^d\times\RR^d;   ~\RR_+)$ the  family of  continuous functions from $\RR^d\times\RR^d$ to $\mathbb{R}_{+}$
satisfying $V(x,x)=0$ for all $x\in \RR^d$.
Throughout this work, $L$ denotes  a generic positive constant which may take different values at  different appearances.

Let $\mathfrak{B}(C)$ be the Borel algebra of $C$ and $\mathcal{P}(C):=\mathcal{P}(C, \mathfrak{B}(C))$ the family of  probability measures on $(C,\mathfrak{B}(C))$. Define the Fortet-Mourier distance $d_{\Xi}$   on  $\mathcal{P}(C)$ \cite[p.8.2]{Dudley} as below,
\begin{align}\la{n2.1}
d_{\Xi}(P_1,P_2)=\sup_{\Psi\in\Xi}\Big\lvert\int_{C}\Psi(X)P_1(\mathrm d X)-\int_{C}\Psi(X)P_2(\mathrm d X)\Big\rvert,~~~~\forall P_1, P_2\in \mathcal P(C),
\end{align}
where $\Xi$ is the test functional space
\begin{align}\la{n2.2}
\Xi:=\Big\{\Psi: C\rightarrow \RR ~\Big\lvert~\lvert\Psi(X_1)-\Psi(X_2)\rvert\leq \|X_1-X_2\|
 ~\hbox{and}~ \sup_{X\in C}\lvert\Psi(X)\rvert\leq 1\Big\}.
\end{align}
\begin{remark}
It is useful to point out that $(\mathcal P(C), d_{\Xi})$ is a complete metric space $($see \textup{\cite[Corollary 10.5]{Dudley}} for details$)$. In addition, the metric $d_{\Xi}$ is equivalent to the Wasserstein distance below
$$\mathbb{W}(P_1,P_2)=\inf_{\pi\in\Pi(P_1,P_2)}\int_{C\times C}(1\wedge\|X_1-X_2\|)\pi(\mathrm dX_1,\mathrm dX_2),$$
where $P_1,~P_2\in \mathcal{P}(C)$ and $\Pi(P_1,P_2)$ is their collection of couplings $($\textup{see \cite[Chapter 6]{Villani}} for details$)$.
\end{remark}
We impose the following  hypotheses.

\textbf{(H1) }For any  $R >0$, there exists a positive
constant $\ell_{R} $ such that
\begin{align*}
  \lvert f(x, y)-f(\bar{x},\bar  y)\rvert\vee \lvert g(x, y)-g(\bar{x}, \bar y)\rvert \leq{}
  \ell_{R} (\lvert x-\bar x\rvert+\lvert y-\bar y\rvert)
 \end{align*}
for any $x,~\bar{x},~y,~\bar y\in \RR ^d $ with  $\lvert x \rvert\vee\lvert\bar{x}\rvert\vee\lvert {y}\rvert\vee\lvert{\bar y}\rvert\leq R$.

\textbf{(H2)} There exist nonnegative constants $\alpha\geq 2$, $a_1, ~a_2,~a_3$ with $a_2>a_3$
 such that
\begin{align*}
\big\langle2x ,f(x,y)\big\rangle+\v g(x,y)\v ^2\leq{}a_1-a_2\v x\v^{\alpha}+a_3\v y\v^{\alpha}
\end{align*}
for any $ x,~y  \in\RR^{d}$.

 \textbf{(H3)} There exist  nonnegative constants { $b_1,~b_2,~b_3, ~b_4$ with $b_1>b_2,~b_3>b_4$, }and  a function $V(\cdot,\cdot)\in {\hat C}(\RR^d\times\RR^d,~\RR_{+})$ such that
\begin{align*}
&2\big\langle x-\bar x ,f(x,y)-f(\bar x,\bar y)\big\rangle+\lvert g(x,y)-g(\bar x,\bar y)\rvert ^2\nn\\
\leq &-b_1\lvert x-\bar x\vert^2+b_2\vert y-\bar y\vert ^2-b_3 V(x,\bar x)+b_4 V(y,\bar y)
\end{align*}
for any $x,~\bar x,~y,~\bar y\in\RR^d$.

It should be pointed out that under (H1) and (H2),  \textup{SDDE} \eqref{n1.1} with the initial data $\xi\in\mathcal C_{\mathcal F_0}^{\alpha}$ has a unique global solution $x^{\xi}(t)$ for $t\geq -\tau$  (see \cite[Theorem 2.4]{Mao2005} and \cite[p.278, Theorem 7.12]{Yuan-Mao-Hybrid} for details).
Let $\{x_{t}^{\xi}\}_{t\geq 0} $ be the corresponding  segment process, where $x_{t}^{\xi}(\theta):=x^{\xi}(t+\theta)$ for $\theta\in[-\tau,0]$.
For any  $\xi\in\mathcal C_{\mathcal F_0}^{\alpha}$ and $t\geq 0$, { denote by $\mu_t^{\xi}$ the probability measure generated by  $x_t^{\xi}$, namely,
\begin{align}
\mu_t^{\xi}(A)=\PP\big\{\omega\in\Omega: x_{t}^{\xi}\in A\big\},~~~\forall A\in\mathfrak{B}(C).
\end{align}}
Let us cite an ergodicity  result  \cite[Theorem 3.6]{WangYa} to close this section.
\begin{lemma}\la{nth1}\textup{\cite[Theorem 3.6]{WangYa}}
Suppose that \textup{(H1)}--\textup{(H3)} hold. Then
the segment process $\{x_{t}^{\xi}\}_{t\geq0}$ of  \eqref{n1.1}  with the initial data $\xi\in C$ is asymptotically stable in distribution and admits a unique invariant measure $\pi(\cdot)\in\mathcal{P}(C)$. Namely,
$$\lim_{t\rightarrow \infty}d_{\Xi}\big(\mu_t^{\xi}(\cdot),\pi(\cdot)\big)=0, ~~~\forall ~\xi\in C.$$
\end{lemma}
\begin{remark}\la{re1}
For the initial data 
$\xi\in \mathcal C_{\mathcal F_0}^{\alpha}$,
{making use of the Chapman-Kolmogorov identity \textup{\cite[p.18-19]{Mohammed}},}  the result of \textup{Lemma \ref{nth1}} still holds.
\end{remark}

\section{Main results}\label{dcsMR}

In this section, we focus on constructing an appropriate explicit scheme and give the main results on the existence and convergence of the numerical invariant measure.
Due to (H1), we may choose a strictly increasing continuous function $\Phi : [1, \infty)\rightarrow \RR_+$ such that $\Phi(R)\rightarrow \infty$ as $R\rightarrow \infty$ and
 \begin{align}\la{n3.1}
 \sup_{\v x\v\vee\v \bar x\v \vee\v y\v\vee\v \bar y\v\leq R}\frac{\v f(x,y)-f(\bar x,\bar y)\v}{\v x-\bar x\v+1\wedge \v y-\bar y\v} + \frac{\v g(x,y)-g(\bar x,\bar y)\v ^2}{(\v x-\bar x \v+1\wedge \v y-\bar y\v)^2}\leq \Phi(R)
 \end{align}
 for any $R\geq1$.
Define a truncation mapping $\Gamma_{\Phi,\nu}^{\triangle}: \RR^d \rightarrow \RR^d$ by
\begin{align}\la{n3.2}
\Gamma_{\Phi,\nu}^{\triangle}(x)= \Big(\v x\v\wedge \Phi^{-1}\big(K\tr^{-\nu}\big)\Big) \frac{x}{\v x\v},
\end{align}
 where  $\Phi^{-1}$ is the inverse function of $\Phi$, $x/\v x\v=\mathbf{0}$ if $x=\mathbf{0}\in \mathbb{R}^d$, $\nu\in(0,1/3]$, and
 \begin{align}\la{constK}
K:=1\vee\Phi(1)\vee|f(\textbf 0,\textbf 0)|\vee|g(\textbf 0,\textbf 0)|^2.
\end{align}
We may suppose without loss of generality that  $\tr=\tau/N\in(0,1)$ with some integer $N>\tau$.
Let $t_{k}=k\tr$ for $k\geq -N$. Then for any $\xi\in \mathcal C_{\mathcal{F}_0}^{\alpha}$, we define the truncated EM scheme of \textup{SDDE} \eqref{n1.1} by
 \begin{align}\la{n3.5}
\left\{
\begin{array}{llll}
\breve u^{\xi,\tr}(t_{k})= \xi(t_k) ,  ~~~k=-N,\ldots,0,&\\
 u^{\xi,\tr}(t_{k}) =\Gamma_{\Phi,\nu}^{\triangle}(\br u^{\xi,\tr}(t_{k})),~~~k=-N,\dots,0, 1, \ldots,&\\
\breve u^{\xi,\tr}(t_{k+1})=u^{\xi,\tr}(t_k)+f(u^{\xi,\tr}(t_k),u^{\xi,\tr}(t_{k-N}))\tr&\\ ~~~~~~~~~~~~~~~~+g(u^{\xi,\tr}(t_k),u^{\xi,\tr}(t_{k-N}))\triangle W_k, ~~~k=0,1,\ldots,& \\
\end{array}
\right.
\end{align}
where  $\triangle W_k=W(t_{k+1})-W(t_{k})$.
Define a piecewise constant process  by
\begin{align}\la{n3.6}
u^{\xi,\tr}(t)=u^{\xi,\tr}(t_k),~~~t\in [t_k, t_{k+1}), ~k\geq -N,
\end{align}
and a piecewise linear continuous  process by
\begin{align}\la{n3.7}
\left\{
\begin{array}{ll}
y^{\xi,\tr}(t)=&\displaystyle \frac{t_{k+1}-t}{\tr} u^{\xi,\tr}(t_{k})+\frac{t-t_{k}}{\tr}u^{\xi,\tr}(t_{k+1}),~~~t\in[t_k,t_{k+1}],~k\geq0,\\
y^{\xi,\tr}(t)=&\Gamma_{\Phi,\nu}^{\triangle}(\xi(t)), ~~~t\in[-\tau,0].
\end{array}
\right.
\end{align}
For any $k\geq0$, let
\begin{align}\la{temlis}
Y_{t_k}^{\xi,\tr}(\theta)=y^{\xi,\tr}(t_k+\theta),~~~\hbox{for~}\theta\in[-\tau,0].
\end{align}
We call  $\{Y_{{t_k}}^{\xi,\tr}\}_{ k\geq 0}$   the  {\it truncated EM segment process (TEMSP)}.

\begin{remark}
In fact, Scheme \eqref{n3.5}  extended from the one in \textup{\cite{Song-Li2021}}. So  by virtue of \textup{\cite{Song-Li2021}} the results also hold for \eqref{n3.5}.  Precisely, under Assumptions  \textup{(H1)} and \textup{(H2)}, for any $\xi\in C$, $u^{\xi,\tr}(\cdot)$ defined by \eqref{n3.6} converges strongly  to the exact solution $x^{\xi}(\cdot)$  of  \eqref{n1.1} in any finite horizon, and reproduces the exponentially stability in infinite horizon when $a_1=0$, $\alpha=2$ given in Assumption \textup{(H2)}. More generally, for any initial data $\xi\in \mathcal C^{\alpha}_{\mathcal F_0}$, 
by  the  techniques in  \textup{\cite{Song-Li2021}} and \textup{Theorem \ref{nth3}} of this paper, these results still hold.
\end{remark}

It is well known that the Markov property  plays a crucial role in investigating the ergodicity.
Since $\{u^{\xi,\tr}(t)\}_{ t\geq-\tau}$ defined by \eqref{n3.6} is not Markovian, it fails to be used directly to approximate the underlying  invariant measure.
Using the similar argument as \cite[Lemma 5.1]{bao-shao-yuan}, we obtain the following result.
\begin{lemma}\la{lemma6}
Suppose  that \textup{(H1)} and \textup{(H2)} hold.
Then for any { $\xi\in C$} and $\tr\in(0,1)$,  TEMSP $\{Y_{t_k}^{\xi,\tr}\}_{k\geq 0}$ defined by \eqref{temlis} is a time-homogenous Markov process, that is, for any $A\in\mathfrak{B}(C)$, $0\leq i<k$, and $\eta\in C$,
$$\PP\big(Y_{t_k}^{\xi,\tr}\in A\v \mathcal{F}_{t_i}\big)=\PP\big(Y_{t_k}^{\xi,\tr}\in A\v Y_{t_i}^{\xi,\tr}\big),$$
$$\PP\big(Y_{t_k}^{\xi,\tr}\in A\v Y_{t_i}^{\xi,\tr}=\eta\big)=\PP\big(Y_{t_{k-i}}^{\eta,\tr}\in A\big).$$
\end{lemma}

For any $\xi\in\mathcal C^{\alpha}_{\mathcal F_0}$,  $\tr\in(0,1)$, and $k\geq 0$, define
\begin{align}\la{n4.1}
\mu^{\xi,\tr}_{t_k}(A)=\PP\big\{\omega\in \Omega: Y_{t_k}^{\xi,\tr}\in A\big\},~~~\forall A\in\mathfrak{B}(C).
\end{align}
It is worth to point out that for any $M>0$,
there exists a $\tr^{*}_M\in(0,1)$ sufficiently small  such that
 \begin{align}\la{initialC}
M\leq\Phi^{-1}\big(K(\tr^{*}_M)^{-\nu}\big),
 \end{align}
which implies that  $\Gamma_{\Phi,\nu}^{\triangle }(x)=x$ for any $x\in\RR^d$ with $\v x\v\leq M$ and $\tr\in(0,\tr^{*}_M]$.
Clearly, in view of \textup{(H2)}, there exists a $\hat\tr_1\in(0,1)$ sufficiently small such that
\begin{align}\la{tr1}
a_2-6K^2 \hat\tr_1^{1-2\nu}>a_3.
\end{align}
In view of \textup{(H3)}, there exists a $\hat\tr_2\in(0,1)$ sufficiently small such that
\begin{align}\la{tr4}
b_1- 4K^2\hat\tr_2^{1-2\nu}>b_2.
\end{align}
In what follows, we  state our main results in this paper.
\begin{theorem}\la{nth3}
Suppose  that \textup{(H1)}--\textup{(H3)} hold.   Let $\hat\tr=\hat\tr_1\wedge\hat\tr_2$.
Then for any $\xi\in\mathcal C^{\alpha}_{\mathcal F_0}$ and $\tr\in(0,\hat\tr]$,   \textup{TEMSP} $\{Y_{t_k}^{\xi,\tr}\}_{k\geq 0}$  defined by \eqref{temlis} is asymptotically stable in distribution and  admits a unique numerical invariant measure $\pi^{\tr}(\cdot)\in \mathcal P(C)$ satisfying
\begin{align}\la{th3+1}
\lim_{t_k\rightarrow \infty}d_{\Xi}\big(\mu^{\xi,\tr}_{t_k}(\cdot),\pi^{\tr}(\cdot)\big)=0,~~~\hbox{uniformly~in
}~ \tr\in(0,\hat\tr].
\end{align}
Moreover, for any $ M>0$, this  convergence is also uniform for $\xi\in \mathbb{B}(  M,\alpha)$.
\end{theorem}
\begin{theorem}\la{nth4}
Suppose that \textup{(H1)}--\textup{(H3)} hold. Then
$$\lim_{\tr\rightarrow 0}d_{\Xi}\big(\pi(\cdot),\pi^{\tr}(\cdot)\big)=0,$$
where $\pi(\cdot)$  and $\pi^{\tr}(\cdot)$ are the underlying invariant measure and the numerical one, respectively.
\end{theorem}

\section{Proofs of theorems}\la{proofs}
Since the  proofs of Theorem \ref{nth3}--\ref{nth4}  are rather technical,  we prepare several notations and lemmas,  and then complete the proofs.
For any $\xi\in \mathcal C_{\mathcal{F}_0}^{\alpha}$, for short, we  write
\begin{align}\la{fg}
&F_t^{\xi,\tr}:=f(u^{\xi,\tr}(t),u^{\xi,\tr}(t-\tau)),
~G_{t}^{\xi,\tr}:=g(u^{\xi,\tr}(t),u^{\xi,\tr}(t-\tau)) ~~~\forall ~t\geq0,
\end{align}
where $u^{\xi,\tr}(t)$ is defined by \eqref{n3.6}.
It follows from  \eqref{n3.1}, \eqref{n3.2}, and \eqref{n3.5} that for any $\xi,\zeta\in \mathcal C_{\mathcal{F}_0}^{\alpha}$ and $t\geq0$,
\begin{align}\la{n3.3}
&\big\v F_{t}^{\xi,\tr}-F_{t}^{\zeta,\tr}\big\v\nn\\
\leq&\Phi\Big(\Phi^{-1}( K\tr^{-\nu})\Big)\Big(\v u^{\xi,\tr}(t)-u^{\zeta,\tr}(t)\v+1\wedge\v u^{\xi,\tr}(t-\tau)-u^{\zeta,\tr}(t-\tau)\v\Big)\nn\\
\leq& K\tr^{-\nu}\Big(\v u^{\xi,\tr}(t)-u^{\zeta,\tr}(t)\v+\v u^{\xi,\tr}(t-\tau)-u^{\zeta,\tr}(t-\tau)\v\Big),
\end{align}
This, along with \eqref{constK}, implies that
\begin{align}\la{n3.4}
\big\v F_{t}^{\xi,\tr}\big\v
\leq K\tr^{-\nu}\big(1+\v u^{\xi,\tr}(t)\v+\v u^{\xi,\tr}(t-\tau)\v\big).
\end{align}
Similarly,
\begin{align}\la{linearG}
\big\v G_{t}^{\xi,\tr}\big\v
\leq K^{\frac{1}{2}}\tr^{-\frac{\nu}{2}}\big(1+\v u^{\xi,\tr}(t)\v+\v u^{\xi,\tr}(t-\tau)\v\big).
\end{align}
For convenience, we define  an  auxiliary  process
\begin{align}\la{n3.8}
\left\{
\begin{array}{lll}
z^{\xi,\tr}(t)=&\displaystyle u^{\xi,\tr}(t_k)+\int_{t_k}^t F^{\xi,\tr}_s\mathrm ds
+\int_{t_k}^t G^{\xi,\tr}_s\mathrm dW(s),~~~\forall~t\in[t_k,t_{k+1}),\\
z^{\xi,\tr}(t)=&\Gamma_{\Phi,\nu}^{\triangle}(\xi(t)),~~~\forall~t\in[-\tau,0].
\end{array}
\right.
\end{align}
To show the uniform boundedness of the norm of TEMSP $\{Y_{t_k}^{\xi,\tr}\}_{k\geq 0}$ in probability we begin with the moment analysis of the numerical solutions.
\begin{lemma}\la{nle1}
Suppose that $(\textup{H}1)$ and $(\textup{H}2)$ hold.
Then for any $M>0$,
\begin{align}\la{n3.9}
\sup_{\tr\in(0, \hat\tr_1]}\sup_{k\geq -N}\sup_{\xi\in B(M)}\E \v u^{\xi,\tr}(t_k)\v ^2\leq L_1,
\end{align}
where $L_1$ is a constant dependent  on $\hat\tr_1$ and $M$.
Moreover, there exists a constant $\bar a=\bar a(\hat\tr_1)\in(0,1]$ such that for any $\tr\in(0,\hat\tr_1]$, $k\geq0$, and $\xi\in B(M)$,
\begin{align}\la{n3.10}
\tr\sum_{i=0}^{k}e^{\bar at_{i+1}}\E\v u^{\xi,\tr}(t_{i})\v ^{\alpha}
\leq L_2(1+e^{\bar at_{k+1}}),
\end{align}
where  $L_2$ is a constant dependent on $\hat\tr_1$ and $M$.
\end{lemma}
\begin{proof}
Fix an $M>0$ and let $\xi\in B(M)$.
For any integer $i\geq0$,
it follows from (H2), \eqref{n3.5} and \eqref{n3.4}  that
\begin{align}\la{n3.12}
&\E\big(\v u^{\xi,\tr}(t_{i+1})\v^2\v\mathcal{F}_{t_i}\big)\leq\E\big(\v\br u^{\xi,\tr}(t_{i+1})\v^2\v\mathcal{F}_{t_i}\big)\nn\\
=&\E\big(\v u^{\xi,\tr}(t_{i})+F^{\xi,\tr}_{t_i}\tr+G^{\xi,\tr}_{t_i}\tr W_{i}\v^2\v\mathcal{F}_{t_i}\big)\nn\\
=&\v u^{\xi,\tr}(t_{i})\v^2+2\langle u^{\xi,\tr}(t_i), F^{\xi,\tr}_{t_i}\rangle\tr+\v G^{\xi,\tr}_{t_i}\v^2\tr+\v F^{\xi,\tr}_{t_i}\v^2\tr^2\nn\\
\leq &\v u^{\xi,\tr}(t_{i})\v^2+a_1\tr-a_2\v u^{\xi,\tr}(t_i)\v^{\alpha}\tr
+a_3\v u^{\xi,\tr}(t_{i-N})\v^{\alpha}\tr\nn\\
&+K^2(1+\v u^{\xi,\tr}(t_i)\v+\v u^{\xi,\tr}(t_{i-N})\v)^2\tr^{2-2\nu}\nn\\
\leq&\v u^{\xi,\tr}(t_{i})\v^2+(a_1+9K^2\tr^{1-2\nu})\tr
-(a_2-3K^2\tr^{1-2\nu})\v u^{\xi,\tr}(t_i)\v^{\alpha}\tr\nn\\
&+(a_3+3K^2\tr^{1-2\nu})\v u^{\xi,\tr}(t_{i-N})\v^{\alpha}\tr,
\end{align}
where  the last inequality   uses $\v u^{\xi,\tr}(t_i)\v^2\leq 1+\v u^{\xi,\tr}(t_i)\v^{\alpha}$.
An application of Lagrange's mean value theorem derives that for any $a>0$, there exists a $ \varsigma\in(t_i,t_{i+1})$ such that
$ e^{at_{i+1}}-e^{at_i}=e^{a\varsigma}a\tr$, which implies
$e^{at_{i+1}}\leq e^{at_{i}}+ e^{at_{i+1}}a\tr$.
Taking expectations in \eqref{n3.12} and using the above inequality and $\tr\in(0,1)$, $\nu\in(0,1/3]$, we arrive at
\begin{align*}
&e^{at_{i+1}}\E\v u^{\xi,\tr}(t_{i+1})\v^2\leq e^{at_{i+1}}\E\Big(\E\big(\v\br u^{\xi,\tr}(t_{i+1})\v^2\v\mathcal{F}_{t_i}\big)\Big)\nn\\
\leq&(e^{at_{i}}+e^{at_{i+1}}a\tr)\E\v u^{\xi,\tr}(t_i)\v^2+e^{at_{i+1}}(a_1+9K^2)\tr\nn\\
&-(a_2-3K^2\tr^{1-2\nu})e^{at_{i+1}}\E\v u^{\xi,\tr}(t_i)\v^{\alpha}\tr\nn\\
&+(a_3+3K^2\tr^{1-2\nu})e^{at_{i+1}}\E \v u^{\xi,\tr}(t_{i-N})\v^{\alpha}\tr\nn\\
\leq&e^{at_{i}}\E\v u^{\xi,\tr}(t_i)\v^2+e^{at_{i+1}}(a_1+9K^2+a)\tr\nn\\
&-(a_2-3K^2\tr^{1-2\nu}-a)
e^{at_{i+1}}\E\v u^{\xi,\tr}(t_i)\v^{\alpha}\tr\nn\\
&+(a_3+3K^2\tr^{1-2\nu})e^{at_{i+1}}\E \v u^{\xi,\tr}(t_{i-N})\v^{\alpha}\tr.
\end{align*}
Summing the above inequality on both sides from $0$ to $k$ derives
\begin{align}\label{eq+1}
&e^{at_{k+1}}\E\v u^{\xi,\tr}(t_{k+1})\v^2\nn\\
\leq&\v u^{\xi,\tr}(0)\v^2+(a_1+9K^2+a)\tr\sum_{i=0}^{k}e^{at_{i+1}}\nn\\
&-(a_2-3K^2\tr^{1-2\nu}-a)\tr\sum_{i=0}^{k}e^{at_{i+1}}\E \v u^{\xi,\tr}(t_i)\v^{\alpha}\nn\\
&+(a_3+3K^2\tr^{1-2\nu})\tr\sum_{i=0}^{k}e^{at_{i+1}}\E\v u^{\xi,\tr}(t_{i-N})\v^{\alpha}\nn\\
\leq& \|\xi\|^2+(a_1+9K^2+a)\tr\frac{e^{ a t_{k+2}}\nn-e^{a\tr}}{e^{ a\tr}-1}\nn\\
&-(a_2-3 K^2\tr^{1-2\nu}-a)\tr\sum_{i=0}^{k}e^{at_{i+1}}\E \v u^{\xi,\tr}(t_i)\v^{\alpha}+(a_3+3K^2\tr^{1-2\nu})\tau e^{a\tau}\|\xi\|^{\alpha}\nn\\
&+(a_3+3K^2\tr^{1-2\nu})e^{a\tau}\tr\sum_{i=0}^{k}e^{at_{i+1}}\E\v u^{\xi,\tr}(t_{i})\v^{\alpha}.
 \end{align}
By virtue of \eqref{tr1} we further  choose an $\bar a \in(0,1]$ sufficiently small  such that
\begin{align}\la{3.15nn}
c :=a_2-3K^2\hat\tr_1^{1-2\nu}-\bar a-(a_3+3K^2\hat\tr_1^{1-2\nu}) e^{\bar a\tau}>0.
\end{align}
Taking $a=\bar a $ in \eqref{eq+1}   yields that for any  $\tr\in(0,\hat\tr_1]$
\begin{align}\la{n3.13}
e^{\bar at_{k+1}}\E\v u^{\xi,\tr}(t_{k+1}\v^2
\leq & \|\xi\|^{2}+(a_3+3K^2)\tau e^{\bar a \tau}\|\xi\|^{\alpha}+(a_1+9K^2+\bar a)\frac{e^{\bar at_{k+2}}}{\bar a}\nn\\
&-c \tr\sum_{i=0}^{k}e^{\bar at_{i+1}}\E \v u^{\xi,\tr}(t_{i})\v^{\alpha}.
\end{align} 
A direct computation derives
\begin{align*}
  \E\v u^{\xi,\tr}(t_{k+1})\v^2
\leq & M^{2} +(a_3+3K^2)\tau e^{\bar a \tau}M^{\alpha} +(a_1+9K^2+\bar a)\frac{e^{ \bar a\triangle}}{\bar a},
\end{align*}
which implies that  \eqref{n3.9} holds. Moreover, the other desired assertion \eqref{n3.10} follows from \eqref{n3.13} directly.
\end{proof}

\begin{lemma}\la{nle2}
Suppose that \textup{(H1)} and \textup{(H2)} hold. Then  for any $\tr\in(0,\hat\tr_1]$ and $M>0$,
\begin{align}\la{n3.15}
\sup_{t\geq0}\sup_{\xi\in B(M)}\E\big\v z^{\xi,\tr}(t)-u^{\xi,\tr}(t)\big\v^2\leq L_3\tr^{1-\nu},
\end{align}
where $L_3$ is a constant dependent on $\hat\tr_1$ and $M$.
\end{lemma}
\begin{proof}
Fix an $M>0$.
For any $\xi\in B(M)$ and $t\in[t_k,t_{k+1})$ with $k\geq 0$, using  \eqref{n3.4}--\eqref{n3.8}, and \eqref{n3.9}, we derive
\begin{align*}
 &\E\big\v z^{\xi,\tr}(t)-u^{\xi,\tr}(t)\big\v^2\nn\\
\leq &2\E\big\v F^{\xi,\tr}_{t_k}\big\v^2\tr^2+2\E\big(\v G^{\xi,\tr}_{t_k}\v^2\v W(t)-W(t_k)\v^2\big)\nn\\
\leq&2K^2\E\big(1+\v u^{\xi,\tr}(t_k)\v+\v u^{\xi,\tr}(t_{k-N})\v\big)^2\tr^{2-2\nu}\nn\\
&~~+2K\E\big(1+\v u^{\xi,\tr}(t_k)\v+\v u^{\xi,\tr}(t_{k-N})\v\big)^2\tr^{1-\nu}\nn\\
\leq&6K(K +1)\big(1+\E\v u^{\xi,\tr}(t_k)\v^2+\E\v u^{\xi,\tr}(t_{k-N})\v^2\big)\tr^{1-\nu}.\end{align*} 
Then the desired assertion follows.
\end{proof}
\begin{lemma}\la{nle3}
Suppose that \textup{(H1)} and \textup{(H2)} hold.
Then for any $M>0$,  $\e>0$, $T>0$,   there exists a $\tr_3=\tr_3(M,\e, T\in(0,\hat\tr_{1}\wedge \tr^*_M]$ such that
  \begin{align}\la{n3.16}
 \sup_{\tr\in(0, \tr_3]}\sup_{k\geq -N}\sup_{\xi\in B(M)} \PP\Big\{\sup_{s\in[t_k,t_k+T]}\v z^{\xi,\tr}(s)\v>\Phi^{-1}(K\tr_{3}^{-\nu})\Big\}<\varepsilon.
   \end{align}
\end{lemma}
\begin{proof}
Our analysis uses a localization procedure.
For any $\eta\in \mathcal C_{\mathcal F_0}^{\alpha}$, $\bar \tr\in(0,1)$, $\tr\in(0,\bar \tr]$ and  $k\geq-N$, define
\begin{align}\la{3.20n}
 \beta_{\bar \tr,k}^{\eta,\tr}=\inf\Big\{s\geq t_{k}: \v z^{\eta,\tr}(s)\v> \Phi^{-1}(K\bar\tr^{-\nu}) \Big\}.
\end{align}
Fix an $M>0$.  
Let $\xi\in B(M)$, $\tr_3\in(0,\hat\tr_1\wedge \tr^*_M]$, $\tr\in(0,\tr_3], k\geq0$.
We should point out that for any $t\in[t_k,\beta_{\tr_3,k}^{\xi,\tr}]$,
\begin{align}\la{contiz}
z^{\xi,\tr}(t)=\displaystyle u^{\xi,\tr}(t_k)+\int_{t_k}^t F^{\xi,\tr}_s\mathrm ds
+\int_{t_k}^t G^{\xi,\tr}_s\mathrm dW(s).
\end{align}
Using the It\^{o} formula,  we obtain from  \eqref{n3.9} and \eqref{contiz}  that for any  $T>0$,
\begin{align}\la{n3.18}
&\E\Big(e^{\bar a((t_k+T)\wedge \beta_{\tr_3,k}^{\xi,\tr}) }\v z^{\xi,\tr}((t_k+T)\wedge \beta_{\tr_3,k}^{\xi,\tr})\v^2\Big)\nn\\
=&e^{\bar a t_k} \E\v u^{\xi,\tr}(t_k)\v^2+\E\int_{t_k}^{(t_k+T)\wedge \beta_{\tr_3,k}^{\xi,\tr}}e^{\bar a s}\Big(\bar a\v z^{\xi,\tr}(s)\v^2+2\langle z^{\xi,\tr}(s),F^{\xi,\tr}_{s}\rangle\nn\\
&+\v G^{\xi,\tr}_{s}\v^2\Big)
\mathrm ds
\leq L_1e^{\bar a t_k}+I_1+I_2+I_3,
\end{align}
where $\bar a\in(0,1]$ is given by Lemma  \ref{nle1}, and
\begin{align*}
&I_1:=\E\int_{t_k}^{(t_k+T)\wedge \beta_{\tr_3,k}^{\xi,\tr}}e^{\bar a s}\Big(\bar a\v u^{\xi,\tr}(s)\v^2+2\langle u^{\xi,\tr}(s),F^{\xi,\tr}_{s}\rangle+\v G^{\xi,\tr}_{s}\v^2\Big)\mathrm ds,\\
&I_2:=\E\int_{t_k}^{(t_k+T)\wedge \beta_{\tr_3,k}^{\xi,\tr}}\bar ae^{\bar a s}\big(\v z^{\xi,\tr}(s)\v^2-\v u^{\xi,\tr}(s)\v^2\big)\mathrm ds,\\
&I_3:=\E\int_{t_k}^{(t_k+T)\wedge \beta_{\tr_3,k}^{\xi,\tr}}2e^{\bar a s}\langle z^{\xi,\tr}(s)-u^{\xi,\tr}(s),F^{\xi,\tr}_{s}\rangle\mathrm ds.
\end{align*}
We start with estimating $I_1$.
Using \textup{(H2)} and together with \eqref{n3.6} leads to
\begin{align*}
I_1
\leq &\E\int_{t_k}^{(t_k+T)\wedge \beta_{\tr_3,k}^{\xi,\tr}}e^{\bar a s}\Big(\bar a\v u^{\xi,\tr}(s)\v^2
+a_1-a_2\v u^{\xi,\tr}(s)\v^{\alpha}+a_3\v u^{\xi,\tr}(s-\tau)\v^{\alpha}\Big)\mathrm ds\nn\\
\leq&(\bar a+a_1)Te^{\bar a(t_k+T)}+(\bar a-a_2+a_3e^{\bar a\tau})\E\int_{t_k}^{(t_k+T)\wedge \beta_{\tr_3,k}^{\xi,\tr}}e^{\bar a s}\v u^{\xi,\tr}(s)\v^{\alpha}\mathrm ds\nn\\
&+a_3e^{\bar a \tau}\E\int_{t_k-\tau}^{t_k}e^{\bar a s}\v u^{\xi,\tr}(s)\v^{\alpha}\mathrm ds\nn\\
\leq &(\bar a+a_1)Te^{\bar a(t_k+T)}+(\bar a-a_2+a_3e^{\bar a\tau})\E\int_{t_k}^{(t_k+T)\wedge \beta_{\tr_3,k}^{\xi,\tr}}e^{\bar a s}\v u^{\xi,\tr}(s)\v^{\alpha}\mathrm ds\nn\\
&+a_3e^{\bar a \tau}\tr\sum_{i=k-N}^{k-1}e^{\bar a t_{i+1}}\E \v u^{\xi,\tr}(t_i)\v^{\alpha}.
\end{align*}
It is straightforward to see  from \eqref{3.15nn} that  $\bar a-a_2+a_3e^{\bar a\tau}<0$.  This, along with \eqref{n3.10} implies that
\begin{align}\la{n3.19}
I_1\leq(\bar a+ a_1)Te^{\bar a(t_k+T)}+a_3e^{\bar a \tau}L_2(1+e^{\bar at_{k}})\leq R_1 e^{\bar a (t_k+T)},
\end{align}
where $R_1:=(1+a_1)T+2a_3e^{\tau}L_2$.
Next we aim to estimate $I_2$.
According to \eqref{n3.9} and \eqref{n3.15} yields
\begin{align}\la{n3.20}
I_2
\leq &\E\int_{t_k}^{(t_k+T)\wedge \beta_{\tr_3,k}^{\xi,\tr}}\bar ae^{\bar a s}\big(2\v z^{\xi,\tr}(s)-u^{\xi,\tr}(s)\v^2+2\v u^{\xi,\tr}(s)\v^2-\v u^{\xi,\tr}(s)\v^2\big)\mathrm ds\nn\\
\leq&\bar ae^{\bar a (t_k+T)}\int_{t_k}^{t_k+T}\big(2\E\v z^{\xi,\tr}(s)-u^{\xi,\tr}(s)\v^2+\E\v u^{\xi,\tr}(s)\v^2\big)\mathrm ds
\leq R_2e^{\bar a (t_k+T)},
\end{align}
where $R_2:=T(2L_3+L_1)$.
Finally, by virtue of \eqref{n3.4}, \eqref{n3.9}, \eqref{n3.15}, $\nu\in(0,1/3]$, and then using the H\"older inequality,  we  get
\begin{align}\la{n3.21}
I_3
\leq&2e^{\bar a (t_k+T)}\int_{t_k}^{t_k+T}\E\big(\v z^{\xi,\tr}(s)-u^{\xi,\tr}(s)\v\v F^{\xi,\tr}_s\v\big)\mathrm ds\nn\\
\leq&2e^{\bar a (t_k+T)}K\tr^{-\nu}\int_{t_k}^{t_k+T}\Big(\E\v z^{\xi,\tr}(s)-u^{\xi,\tr}(s)\v^2
\E\big(1+\v u^{\xi,\tr}(s)\v\nn\\
&+\v u^{\xi,\tr}(s-\tau)\v\big)^2\Big)^{\frac{1}{2}}\mathrm ds\nn\\
\leq &2\sqrt3KT(L_3(1+2L_1)\tr^{1-3\nu})^{1/2}e^{\bar a (t_k+T)}\leq R_3e^{\bar a (t_k+T)},
\end{align}
where $R_3:=2\sqrt3KT(L_3(1+2L_1))^{1/2}$.
Plugging \eqref{n3.19}--\eqref{n3.21} back into \eqref{n3.18} gives
\begin{align*}
e^{\bar a t_k}\E\big\v z^{\xi,\tr}\big((t_k+T)\wedge \beta_{\tr_3,k}^{\xi,\tr}\big)\big\v^2\leq&\E\Big(e^{\bar a((t_k+T)\wedge \beta_{\tr_3,k}^{\xi,\tr}) }\v z^{\xi,\tr}((t_k+T)\wedge \beta_{\tr_3,k}^{\xi,\tr})\v^2\Big)\nn\\
\leq &Re^{\bar a (t_k+T)},
\end{align*}
where $R:=L_1+R_1+R_2+R_3$.
This, along with $\bar a\in(0,1]$ implies
\begin{align}\la{stopB}
&\E\big\v z^{\xi,\tr}\big((t_k+T)\wedge \beta_{\tr_3,k}^{\xi,\tr}\big)\big\v^2
\leq Re^{T}.
\end{align}
For any $\e>0$, choose a $\tr_3=\tr_3(M, \e, T)\in(0,\hat\tr_1\wedge \tr^*_M]$  sufficiently small such that
\begin{align}\la{n3.17}
Re^{T}<\e(\Phi^{-1}(K\tr_3^{-\nu}))^2.
\end{align}
According to \eqref{stopB} and \eqref{n3.17} concludes that
\begin{align*}
&\sup_{\tr\in(0, \tr_3]}\sup_{k\geq 0}\sup_{\xi\in B(M)} \PP\Big\{\beta_{\tr_3,k}^{\xi,\tr}< t_k+T\Big\}\nn\\
\leq&\sup_{\tr\in(0, \tr_3]}\sup_{k\geq 0}\sup_{\xi\in B(M)}\frac{\E\v z^{\xi,\tr}((t_k+T)\wedge \beta_{\tr_3,k}^{\xi,\tr})\v^2}{(\Phi^{-1}(K\tr_3^{-\nu}))^{2}}\nn\\
\leq&\sup_{\tr\in(0, \tr_3]}\sup_{k\geq 0}\sup_{\xi\in B(M)} \frac{Re^{ T}}{(\Phi^{-1}(K\tr_3^{-\nu}))^{2}}<\e.
\end{align*}
Note that $\{z^{\xi,\tr}(t)\}_{t\geq0}$ is right-continuous and left-limit, and
$$\lim_{t \uparrow t_k}\v z^{\xi,\tr}(t)\v=\v\breve u(t_k)\v\geq \v u(t_k)\v=\v z^{\xi,\tr}(t_k)\v,~~~\forall k\geq0.$$
This implies that for any $\tr \in(0, \tr_3]$, ~$k\geq0$, and $\xi\in B(M)$,
\begin{align}\la{n3.22}
&\PP\Big\{\sup_{s\in[t_k,t_k+T]}\v z^{\xi,\tr}(s)\v>\Phi^{-1}(K\tr_{3}^{-\nu})\Big\}\nn\\
=&\PP\Big\{\v z^{\xi,\tr}(s)\v>\Phi^{-1}(K\tr_{3}^{-\nu}), \exists s\in[t_k,t_k+T)\Big\}
=\PP\Big\{\beta_{\tr_3,k}^{\xi,\tr}< t_k+T\Big\}<\e.
\end{align}
Therefore \eqref{n3.16} is characterized by  \eqref{initialC} and \eqref{n3.22}.
\end{proof}
Thanks for the above lemmas, we go a further step to analyze the
uniform boundedness of the norm
of TEMSP $\{Y_{t_k}^{\xi,\tr}\}_{k\geq0}$ in probability.
\begin{proposition}\la{leerry}
Suppose that $(\textup{H}1)$ and $(\textup{H}2)$ hold.
Then for any $M>0$, $\e>0$, there exists a  $\Lambda^*=\Lambda^*(\hat\tr_1, M, \e)>M$ such that
\begin{align}\la{conclu4.6}
\sup_{\tr\in(0,\hat\tr_1]}\sup_{k\geq 0}\sup_{\xi\in B(M)}\PP\Big\{ \|Y^{\xi,\tr}_{t_k}\|>\Lambda^*\Big\}<\e.
\end{align}
\end{proposition}
\begin{proof}
Fix an $M>0$.
For any $\e>0$,
making use of Lemma \ref{nle3}, there exists a  $\tr_3=\tr_3(M, \e,\tau)\in(0,\hat\tr_1\wedge\tr_M^*]$ such that
\begin{align}\la{n4.6}
&\sup_{\tr\in(0,\tr_3]}\sup_{k\geq-N}\sup_{\xi\in B(M)}\PP\Big\{\sup_{t\in[t_k,t_{k}+\tau]}\v z^{\xi,\tr}(t)\v>\Phi^{-1}\big(K\tr_3^{-\nu}\big)\Big\}
<\e.
\end{align}
It follows from \eqref{n3.7}, \eqref{temlis}, and \eqref{n3.8} that for any $\tr\in(0,\tr_3]$, $k\geq0$, and $\xi\in B(M)$,
$$\|Y^{\xi,\tr}_{t_k}\|=\sup_{t\in[t_k-\tau,t_k]}\v y^{\xi,\tr}(t)\v\leq M\vee\sup_{k-N\leq i\leq k}\v u^{\xi,\tr}(t_i)\v\leq M\vee\sup_{t\in[t_k-\tau,t_k]}\v z^{\xi,\tr}(t)\v.$$
Letting $\Lambda^*=\Phi^{-1}\big(K\tr_3^{-\nu}\big)$.
Combining the above inequality with \eqref{n4.6} and using $\Lambda^*\geq M$ we derive
\begin{align}\la{4.8n}
&\sup_{\tr\in(0,\tr_3]}\sup_{k\geq0}\sup_{\xi\in B(M)}\PP\Big\{\|Y^{\xi,\tr}_{t_k}\|>\Lambda^*\Big\}\nn\\
\leq& \sup_{\tr\in(0,\tr_3]}\sup_{k\geq0}\sup_{\xi\in B(M)}\PP\Big\{\sup_{t\in[t_k-\tau,t_{k}]}\v z^{\xi,\tr}(t)\v>\Phi^{-1}\big(K\tr_3^{-\nu}\big)\Big\}<\e.
\end{align}
It follows from the truncation property in \eqref{n3.5}--\eqref{n3.7}  that
\begin{align*}
&\sup_{\tr\in[\tr_3,\hat\tr_1]}\sup_{k\geq0}\sup_{\xi\in B(M)}\|Y^{\xi,\tr}_{t_k}\|
\leq M\vee\sup_{\tr\in[\tr_3,\hat\tr_1]}\Phi^{-1}\big(K\tr^{-\nu}\big)\leq\Lambda^*.
\end{align*}
This, along with \eqref{4.8n} implies
\begin{align*}
&\sup_{\tr\in(0,\hat\tr_1]}\sup_{k\geq 0}\sup_{\xi\in B(M)}\PP\big\{\|Y^{\xi,\tr}_{t_k}\|>\Lambda^*\big\}\nn\\
\leq&\sup_{\tr\in(0,\tr_3]}\sup_{k\geq 0}\sup_{\xi\in  B(M)}\PP\big\{\|Y^{\xi,\tr}_{t_k}\|>\Lambda^*\big\}
+\sup_{\tr\in[\tr_3,\hat\tr_1]}\sup_{k\geq 0}\sup_{\xi\in B(M)}\PP\big\{\|Y^{\xi,\tr}_{t_k}\|>\Lambda^*\big\}\nn\\
<&\e.
\end{align*}
The proof is complete.
\end{proof}
Furthermore,  we study the  attraction  of TEMSP $\{Y_{t_k}^{\xi,\tr}\}_{k\geq 0}$ in probability.  The strategy is as follows:
we first  prove the continuity of the sample path $y^{\xi,\tr}(t)$  with respect to $t$ and  the continuity of  $\{y^{\xi,\tr}(t_{k})\}_{k\geq0}$ (equivalent with $\{u^{\xi,\tr}(t_{k})\}_{k\geq0}$ )  with respect to the initial data $\xi$;
Next making use of above results and the fact
\begin{align*}
&\v y^{\xi,\tr}(t_{k}+\theta)-y^{\zeta,\tr}(t_{k}+\theta)\v\nn\\
\leq &\v y^{\xi,\tr}(t_{k}+\theta)-y^{\xi,\tr}(t_{k})\v+\v y^{\xi,\tr}(t_{k})-y^{\zeta,\tr}(t_{k})\v\nn\\
&+\v y^{\zeta,\tr}(t_{k})-y^{\zeta,\tr}(t_{k}+\theta)\v,~~~\forall\theta\in[-\tau,0],
\end{align*}
 we estimate  $\sup_{\theta\in[-\tau,0]}\v y^{\xi,\tr}(t_{k}+\theta)-y^{\zeta,\tr}(t_{k}+\theta)\v$ in probability and obtain the  attraction  of  $\{Y_{t_k}^{\xi,\tr}\}_{k\geq 0}$ in  probability.
\begin{lemma}\la{nole4}
Suppose that $(\textup{H}1)$ and $(\textup{H}2)$ hold.
Then for any $M>0$, $\e_1>0$, $\e_2>0$, there exists a $\tr_4=\tr_4(M, \e_1, \e_2)\in(0,\tr_3]$ such that
\begin{align}\la{n3.27}
&\sup_{\tr\in(0,\tr_4]}\sup_{k\geq0}\sup_{\xi\in B(M)}\PP\Big\{\sup_{t\in[t_k, t_{k}+\tau]}\v y^{\xi,\tr}(t)-z^{\xi,\tr}(t)\v\geq\e_2\Big\}<\e_1,
\end{align}
where $\tr_3=\tr_3(M,\e_1/2,2\tau)$ is given by \textup{Lemma \ref{nle3}}.
\end{lemma}
\begin{proof}
Fix an $M>0$.
For any $\e_1>0$, recalling \eqref{3.20n} and  applying Lemma \ref{nle3}, there is a $\tr_3=\tr_3(M,\e_1/2,2\tau)\in(0,\hat\tr_1\wedge\tr^*_M]$ such that
\begin{align}\la{beta4}
\sup_{\tr\in(0,\tr_3]}\sup_{k\geq-N}\sup_{\xi\in B(M)}\PP\Big\{\beta_{\tr_3,k}^{\xi,\tr}< t_{k}+2\tau\Big\}<\frac{\e_1}{2}.
\end{align}
Let $\tr\in(0,\tr_3]$, $k\geq -N$,  $\xi\in B(M)$.
It follows from   \eqref{n3.7} and  \eqref{contiz} that
\begin{align*}
&\E\Big(\sup_{t\in[t_k+\tau,t_{k}+2\tau]}\big(\big\v y^{\xi,\tr}(t)-z^{\xi,\tr}(t)\big\v^4\textbf 1_{\{\beta_{\tr_3,k}^{\xi,\tr}\geq t_k+2\tau\}}\big)\Big)\nn\\
=&\E\Big(\sup_{i\in\{k+N,\cdots,k+2N-1\}}\sup_{t_{i}\leq t\leq t_{i+1}}\big(\big\v \frac{t_{i+1}-t}{\tr}u^{\xi,\tr}(t_{i})+\frac{t-t_i}{\tr}u^{\xi,\tr}(t_{i+1})\nn\\
&-u^{\xi,\tr}(t_{i})-F^{\xi,\tr}_{t_{i}}(t-t_i)-G^{\xi,\tr}_{t_{i}}
\big(W(t)-W(t_{i})\big)\big\v^4\textbf 1_{\{\beta_{\tr_3,k}^{\xi,\tr}\geq t_k+2\tau\}}\big)\Big)\nn\\
=&\E\Big(\sup_{i\in\{k+N,\cdots,k+2N-1\}}\sup_{t_{i}\leq t\leq t_{i+1}}\big(\big\v \frac{t-t_i}{\tr}\big(u^{\xi,\tr}(t_{i+1})-u^{\xi,\tr}(t_{i})\big)\nn\\
&-F^{\xi,\tr}_{t_{i}}(t-t_i)-G^{\xi,\tr}_{t_{i}}
\big(W(t)-W(t_{i})\big)\big\v^4\textbf 1_{\{\beta_{\tr_3,k}^{\xi,\tr}\geq t_k+2\tau\}}\big)\Big).
\end{align*}
By  virtue of \eqref{n3.5} we obtain
\begin{align}\la{err+1}
&\E\Big(\sup_{t\in[t_k+\tau,t_{k}+2\tau]}\big(\big\v y^{\xi,\tr}(t)-z^{\xi,\tr}(t)\big\v^4\textbf 1_{\{\beta_{\tr_3,k}^{\xi,\tr}\geq t_k+2\tau\}}\big)\Big)\nn\\
=&\E\Big(\sup_{i\in\{k+N,\cdots,k+2N-1\}}\sup_{t_{i}\leq t\leq t_{i+1}}\big(\big\v\frac{t-t_i}{\tr}\big(F^{\xi,\tr}_{t_{i}}\tr+G^{\xi,\tr}_{t_{i}}\tr W_{i}\big)\nn\\
&-F^{\xi,\tr}_{t_{i}}(t-t_i)-G^{\xi,\tr}_{t_{i}}
\big(W(t)-W(t_{i})\big)\big\v^4\textbf 1_{\{\beta_{\tr_3,k}^{\xi,\tr}\geq t_k+2\tau\}}\big)\Big)\nn\\
\leq&8\E\Big(\sup_{i\in\{k+N,\cdots,k+2N-1\}}\sup_{t_{i}\leq t\leq t_{i+1}}\big(\big\v\frac{t-t_i}{\tr}G^{\xi,\tr}_{t_{i}}\tr W_i\big\v^4\textbf 1_{\{\beta_{\tr_3,k}^{\xi,\tr}\geq t_k+2\tau\}}\big)\Big)\nn\\
&+8\E\Big(\sup_{i\in\{k+N,\cdots,k+2N-1\}}\sup_{t_{i}\leq t\leq t_{i+1}}\big(\big\v G^{\xi,\tr}_{t_{i}}(W(t)-W(t_i))\big\v^4\textbf 1_{\{\beta_{\tr_3,k}^{\xi,\tr}\geq t_k+2\tau\}}\big)\Big).
\end{align}
In view of (H1), there exists a constant $L>0$ sufficiently large such that
\begin{align*}
\sup_{\v x\v\vee\v y\v\leq \Phi^{-1}(K\tr_3^{-\nu})}\v g(x,y)\v\leq L.
\end{align*}
Inserting this into \eqref{err+1} and using the Doob martingale inequality implies
\begin{align*}
&\E\Big(\sup_{t\in[t_k+\tau,t_{k}+2\tau]}\big(\big\v y^{\xi,\tr}(t)-z^{\xi,\tr}(t)\big\v^4\textbf 1_{\{\beta_{\tr_3,k}^{\xi,\tr}\geq t_k+2\tau\}}\big)\Big)\nn\\
\leq &8L^4\E\Big(\sup_{i\in\{k+N,\cdots,k+2N-1\}}\big\v\tr W_i\big\v^4\Big)\nn\\&+8L^4\E\Big(\sup_{i\in\{k+N,\cdots,k+2N-1\}}\sup_{t_{i}\leq t\leq t_{i+1}}\big\v W(t)-W(t_i)\big\v^4\Big)\nn\\
\leq&8L^4\sum_{i=k+N}^{k+2N-1}\E\Big(\big\v\tr W_i\big\v^4+\sup_{t_{i}\leq t\leq t_{i+1}}\big\v  W(t)-W(t_i)\big\v^4\Big)
\leq\tilde LN\tr^{2}\leq \tilde L\tau\tr,
\end{align*}
where $\tilde L$ is a constant.
For any $\e_2>0$, choose a $\tr_4=\tr_4(M, \e_1,\e_2)\in(0,\tr_3]$ sufficiently small such that
\begin{align*}
\frac{\tilde L\tau\tr_4}{\e_2^4}<\frac{\e_1}{2}.
\end{align*}
An application of  Chebyshev's inequality arrives at that for any $\tr\in(0,\tr_4]$,
\begin{align*}
&\PP\Big\{\beta_{\tr_3,k}^{\xi,\tr}\geq t_{k}+2\tau, \sup_{t\in[t_k+\tau,t_{k}+2\tau]}\big\v y^{\xi,\tr}(t)-z^{\xi,\tr}(t)\big\v\geq\e_2\Big\}\nn\\
\leq&\frac{1}{\e_2^4}\E\Big(\sup_{t\in[t_k+\tau,t_{k}+2\tau]}\big(\big\v y^{\xi,\tr}(t)-z^{\xi,\tr}(t)\big\v^4\textbf 1_{\{\beta_{\tr_3,k}^{\xi,\tr}\geq t_k+2\tau\}}\big)\Big)
\leq\frac{\tilde L\tau\tr_4}{\e_2^4}<\frac{\e_1}{2}.
\end{align*}
This, combining with \eqref{beta4} implies that the required assertion \eqref{n3.27} follows.
\end{proof}
\begin{lemma}\la{le8s}
Suppose that \textup{(H1)} and \textup{(H2)} hold. Then
for any $M>0$, $\e_1>0$,~$\e_2>0$, there exists a positive integer  $j^{*}=j^*(M, \e_1,\e_2,\tr_4)$  such that
\begin{align*}
\sup_{\tr\in(0,\tr_4]}\sup_{k\geq 0}\sup_{\xi\in  B(M)}\PP\Big\{\sup_{\substack{\v s_1-s_2\v \leq {\tau/j^*}\\ s_1,s_2\in[t_{k},t_{k}+\tau]}}\v y^{\xi,\tr}(s_1)-y^{\xi,\tr}(s_2)\v\geq\e_2\Big\}
<\e_1,
\end{align*}
where $\tr_4=\tr_4(M, \e_1/2,\e_2/3)$ is given by \textup{Lemma \ref{nole4}}.
\end{lemma}
\begin{proof}
Fix an $M>0$.
For any $\e_1>0$,
recalling \eqref{3.20n} and applying Lemma \ref{nle3}, there is a $\tr_3=\tr_3(M, \e_1/4,2\tau)\in(0,\hat\tr_1\wedge\tr_{M}^*]$ such that
\begin{align}\la{beta8s}
\sup_{\tr\in(0,\tr_3]}\sup_{k\geq-N}\sup_{\xi\in  B(M)}\PP\Big\{\beta_{\tr_3,k}^{\xi,\tr}< t_{k}+2\tau\Big\}<\frac{\e_1}{4}.
\end{align}
Let $\tr\in(0,\tr_3]$, $k\geq -N$, $\xi\in B(M)$.
For any integer $j^*\geq1$, define $t_k^{j^*,j}=t_k+j\tau/j^*,~j=0,\cdots, j^*.$
According to  \eqref{contiz}  and the Burkholder-Davis-Gundy inequality,  we derive that for any  $j\in\{0,\cdots, j^*-1\}$,
\begin{align}\la{n4.3}
&\E\Big(\sup_{t\in[t_{k+N}^{j^*,j},t_{k+N}^{j^*,j+1}]}\big(\v z^{\xi,\tr}(t\wedge\beta_{\tr_3,k}^{\xi,\tr} )-z^{\xi,\tr}(t_{k+N}^{j^*,j}\wedge\beta_{\tr_3,k}^{\xi,\tr})\v^4\big)\Big)\nn\\
= &\E\Big(\sup_{t\in[t_{k+N}^{j^*,j},t_{k+N}^{j^*,j+1}]}\big(\big\v\int_{t_{k+N}^{j^*,j}\wedge\beta_{\tr_3,k}^{\xi,\tr}}
^{t\wedge\beta_{\tr_3,k}^{\xi,\tr}}F^{\xi,\tr}_h\mathrm dh+\int_{t_{k+N}^{j^*,j}\wedge\beta_{\tr_3,k}^{\xi,\tr}}^{t\wedge\beta_{\tr_3,k}^{\xi,\tr}}G^{\xi,\tr}_h\mathrm d W(h)\big\v^{4}\big)\Big)\nn\\
\leq &8\E\Big(\sup_{t\in[t_{k+N}^{j^*,j},t_{k+N}^{j^*,j+1}]}\big(\big\v\int_{t_{k+N}^{j^*,j}}^{t}F^{\xi,\tr}_h\textbf 1_{\{\beta_{\tr_3,k}^{\xi,\tr}\geq h\}}\mathrm dh\big\v^4\nn\\
&+\big\v \int_{t_{k+N}^{j^*,j}}^{t}G^{\xi,\tr}_h\textbf 1_{\{\beta_{\tr_3,k}^{\xi,\tr}\geq h\}}\mathrm d W(h)\big\v^4\big)\Big)\nn\\
\leq&8\E\Big(\int_{t_{k+N}^{j^*,j}}^{t_{k+N}^{j^*,j+1}}\big\v F^{\xi,\tr}_h\big\v \textbf 1_{\{\beta_{\tr_3,k}^{\xi,\tr}\geq h\}}\mathrm dh\Big)^4+\frac{2^{21}}{3^6}\E\Big(\int_{t_{k+N}^{j^*,j}}^{t_{k+N}^{j^*,j+1}}\big\v G^{\xi,\tr}_h\big\v^2\textbf 1_{\{\beta_{\tr_3,k}^{\xi,\tr}\geq h\}}\mathrm d h\Big)^2.
\end{align}
By virtue of \textup{(H1)} there exists a constant $L>0$ such that
\begin{align}
\sup_{\v x\v\vee\v y\v\leq \Phi^{-1}(K\tr_3^{-\nu})}\big(\v f(x,y)\v\vee \v g(x,y)\v\big)\leq L.
\end{align}
Inserting this into \eqref{n4.3} implies
\begin{align}\la{n4.4}
&\E\Big(\sup_{t\in[t_{k+N}^{j^*,j},t_{k+N}^{j^*,j+1}]}\big(\v z^{\xi,\tr}(t\wedge\beta_{\tr_3,k}^{\xi,\tr} )-z^{\xi,\tr}(t_{k+N}^{j^*,j}\wedge\beta_{\tr_3,k}^{\xi,\tr})\v^4\big)\Big)\leq \frac{R_4}{(j^*)^2},
\end{align}
where 
$R_4:=8L^4\tau^4+2^{21}L^4\tau^2/3^6$.
 For any $\e_2>0$, choose $j^*\geq 1\vee(4\cdot9^4R_4/(\e_1\e_2^4))$.
The fundamental theory of calculus shows that
\begin{align*}
&\PP\Big\{\beta_{\tr_3,k}^{\xi,\tr}\geq t_{k}+2\tau, \sup_{\substack{\v s_1-s_2\v\leq\tau/j^*\\ s_1,s_2\in[t_{k}+\tau,t_{k}+2\tau]}}\v z^{\xi,\tr}(s_1)-z^{\xi,\tr}(s_2)\v\geq\frac{\e_2}{3}\Big\}\nn\\
\leq&\PP\Big\{\beta_{\tr_3,k}^{\xi,\tr}\geq t_{k}+2\tau, ~3\max_{0\leq j\leq j^*-1}\sup_{t\in\mathcal [t_{k+N}^{j^*,j},t_{k+N}^{j^*,j+1}]}\v z^{\xi,\tr}(t)-z^{\xi,\tr}(t_{k+N}^{ j^*,j})\v\geq\frac{\e_2}{3}\Big\}\nn\\
\leq &\sum_{j=0}^{j^*-1}\PP\Big\{\beta_{\tr_3,k}^{\xi,\tr}\geq t_{k}+2\tau, ~\sup_{t\in\mathcal [t_{k+N}^{j^*,j},t_{k+N}^{j_1,j+1}]}\v z^{\xi,\tr}(t)-z^{\xi,\tr}(t_{k+N}^{ j^*,j})\v\geq\frac{\e_2}{9}\Big\}.
\end{align*}
Using the Chebyshev inequality implies
\begin{align*}
&\PP\Big\{\beta_{\tr_3,k}^{\xi,\tr}\geq t_{k}+2\tau, \sup_{\substack{\v s_1-s_2\v\leq \tau/ j^*\\ s_1,s_2\in[t_{k}+\tau,t_{k}+2\tau]}}\v z^{\xi,\tr}(s_1)-z^{\xi,\tr}(s_2)\v\geq\frac{\e_2}{3}\Big\}\nn\\
\leq&\frac{9^4}{\e_2^4}\sum_{j=0}^{j^*-1}\E\Big(\sup_{t\in\mathcal [t_{k+N}^{j^*,j},t_{k+N}^{j^*,j+1}]}\big(\v z^{\xi,\tr}(t)-z^{\xi,\tr}(t_{k+N}^{j^*,j})\v^4\textbf 1_{\{\beta_{\tr_3,k}^{\xi,\tr}\geq t_{k}+2\tau\}}\big)\Big)\nn\\
\leq&\frac{9^4}{\e_2^4}\sum_{j=0}^{j^*-1}\E\Big(\sup_{t\in\mathcal [t_{k+N}^{j^*,j},t_{k+N}^{j^*,j+1}]}\big(\v z^{\xi,\tr}(t\wedge \beta_{\tr_3,k}^{\xi,\tr})-z^{\xi,\tr}(t_{k+N}^{j^*,j}\wedge \beta_{\tr_3,k}^{\xi,\tr})\v^4\big)\Big).
\end{align*}
This, together with \eqref{n4.4} implies that
\begin{align*}
&\PP\Big\{\beta_{\tr_3,k}^{\xi,\tr}\geq t_{k}+2\tau, \sup_{\substack{\v s_1-s_2\v\leq\tau/ j^*\\ s_1,s_2\in[t_{k}+\tau,t_{k}+2\tau]}}\v z^{\xi,\tr}(s_1)-z^{\xi,\tr}(s_2)\v\geq\frac{\e_2}{3}\Big\}
\leq \frac{9^4R_4 }{\e_2^{4}j^*}\leq\frac{\e_1}{4}.
\end{align*}
Combining the above inequality with \eqref{beta8s} we arrive at
\begin{align}\la{4.6n}
\sup_{\tr\in(0,\tr_3]}\sup_{k\geq 0}\sup_{\xi\in  B(M)}\PP\Big\{\sup_{\substack{\v s_1-s_2\v\leq \tau/ j^*\\ s_1,s_2\in[t_{k},t_{k}+\tau]}}\v z^{\xi,\tr}(s_1)-z^{\xi,\tr}(s_2)\v\geq\frac{\e_2}{3}\Big\}
<\frac{\e_1}{2}.
\end{align}
In view of Lemma \ref{nole4}, there exists a $\tr_4=\tr_4(M, \e_1/2,\e_2/3)\in(0,\tr_3]$ such that
\begin{align}\la{4.7n}
\sup_{\tr\in(0,\tr_4]}\sup_{k\geq0}\sup_{\xi\in B(M)}\PP\Big\{\sup_{t\in[t_k,t_{k}+\tau]}\v y^{\xi,\tr}(t)-z^{\xi,\tr}(t)\v\geq\frac{\e_2}{3}\Big\}
<\frac{\e_1}{2}.
\end{align}
It is straightforward to see from \eqref{4.6n} and \eqref{4.7n} that
\begin{align*}
&\sup_{\tr\in(0,\tr_4]}\sup_{k\geq0}\sup_{\xi\in  B(M)}\PP\Big\{ \sup_{\substack{\v s_1-s_2\v\leq \tau/j^*\\ s_1,s_2\in[t_{k},t_{k}+\tau]}}\big\v y^{\xi,\tr}(s_1)-y^{\xi,\tr}(s_2)\big\v\geq\e_2\Big\}\nn\\
\leq& \sup_{\tr\in(0,\tr_4]}\sup_{k\geq0}\sup_{\xi\in  B(M)}\PP\Big\{ 2\sup_{t\in[t_{k},t_{k}+\tau]}\big\v y^{\xi,\tr}(t)-z^{\xi,\tr}(t)\big\v\geq\frac{2\e_2}{3}\Big\}\nn\\
&+\sup_{\tr\in(0,\tr_4]}\sup_{k\geq0}\sup_{\xi\in  B(M)}\PP\Big\{\sup_{\substack{\v s_1-s_2\v\leq\tau/ j^*\\ s_1,s_2\in[t_{k},t_{k}+\tau]}}\big\v z^{\xi,\tr}(s_1)-z^{\xi,\tr}(s_2)\big\v\geq\frac{\e_2}{3}\Big\}
<\e_1.
\end{align*}
The proof is complete.
\end{proof}
\begin{lemma}\la{4.4nle}
Suppose that \textup{(H3)} holds.
Then  there is a $ \bar\lambda=\bar\lambda(\hat\tr_2)\in(0,1]$ such that for any  $M>0$, $\tr\in(0,\hat\tr_2]$, and $k\geq 0$,
 \begin{align*}
\sup_{\xi,\zeta\in B(M)}\E\v u^{\xi,\tr}(t_{k})-u^{\zeta,\tr}(t_{k})\v^2\leq L_{M}e^{-\bar \lambda t_{k}},
\end{align*}
where 
$$L_{M}:=4M^2(1+(b_2+2K^2)\tau e^{\tau})+b_4\tau e^{\tau}\sup_{\v x\v\vee\v y\v\leq M}V(x,y).$$
\end{lemma}
\begin{proof}
Fix an $M>0$.
For any  $\tr\in(0,\hat\tr_2]$, $i\geq 0$, and $\xi,\zeta\in B(M)$,
one observes from \eqref{n3.5} that
\begin{align*}
&\v\br u^{\xi,\tr}(t_{i+1})-\br u^{\zeta,\tr}(t_{i+1})\v^2\nn\\
=&\v u^{\xi,\tr}(t_{i})-u^{\zeta,\tr}(t_{i})\v^2+2\langle u^{\xi,\tr}(t_{i})-u^{\zeta,\tr}(t_{i}),F^{\xi,\tr}_{t_i}-F^{\zeta,\tr}_{t_i}\rangle\tr\nn\\
&+\v(G^{\xi,\tr}_{t_i}-G^{\zeta,\tr}_{t_i})\tr W_{i}\v^2+\v F^{\xi,\tr}_{t_i}-F^{\zeta,\tr}_{t_i}\v^2\tr^2\nn\\
&+2\langle u^{\xi,\tr}(t_{i})-u^{\zeta,\tr}(t_{i}),(G^{\xi,\tr}_{t_i}-G^{\zeta,\tr}_{t_i})\tr W_{i}\rangle\nn\\
&+2\langle F^{\xi,\tr}_{t_i}-F^{\zeta,\tr}_{t_i}, (G^{\xi,\tr}_{t_i}-G^{\zeta,\tr}_{t_i})\tr W_{i}\rangle\tr.
\end{align*}
Taking expectations in both sides  of the above inequality , and by using  \textup{(H3)}, and \eqref{n3.3}, leads to
\begin{align}\la{n4.8}
&\E\v \br u^{\xi,\tr}(t_{i+1})-\br u^{\zeta,\tr}(t_{i+1})\v^2\nn\\
=&\E\v u^{\xi,\tr}(t_{i})-u^{\zeta,\tr}(t_{i})\v^2+2\E\langle u^{\xi,\tr}(t_{i})-u^{\zeta,\tr}(t_{i}),F^{\xi,\tr}_{t_i}-F^{\zeta,\tr}_{t_i}\rangle\tr\nn\\
&+\E\v(G^{\xi,\tr}_{t_i}-G^{\zeta,\tr}_{t_i})\v^2\tr+\E\v F^{\xi,\tr}_{t_i}-F^{\zeta,\tr}_{t_i}\v^2\tr^2\nn\\
\leq &\E\v u^{\xi,\tr}(t_{i})-u^{\zeta,\tr}(t_{i})\v^2
-b_1\E\v u^{\xi,\tr}(t_{i})-u^{\zeta,\tr}(t_{i})\v^2\tr\nn\\
&+b_2\E\v u^{\xi,\tr}(t_{i-N})-u^{\zeta,\tr}(t_{i-N})\v ^2\tr
-b_3\E V\big(u^{\xi,\tr}(t_{i}),u^{\zeta,\tr}(t_{i})\big)\tr\nn\\
&+b_4\E V\big(u^{\xi,\tr}(t_{i-N}),u^{\zeta,\tr}(t_{i-N})\big)\tr\nn\\
&+K^2\E\big(\v u^{\xi,\tr}(t_{i})-u^{\zeta,\tr}(t_{i})\v
+\v u^{\xi,\tr}(t_{i-N})-u^{\zeta,\tr}(t_{i-N})\v\big)^2\tr^{2-2\nu}\nn\\
\leq&\E\v u^{\xi,\tr}(t_{i})-u^{\zeta,\tr}(t_{i})\v^2
-(b_1-2K^2\tr^{1-2\nu})\E\v u^{\xi,\tr}(t_{i})-u^{\zeta,\tr}(t_{i})\v^2\tr\nn\\
&+(b_2+2K^2\tr^{1-2\nu})\E\v u^{\xi,\tr}(t_{i-N})-u^{\zeta,\tr}(t_{i-N})\v^2\tr\nn\\
&-b_3\E V\big(u^{\xi,\tr}(t_{i}),u^{\zeta,\tr}(t_{i})\big)\tr+b_4\E V\big(u^{\xi,\tr}(t_{i-N}),u^{\zeta,\tr}(t_{i-N})\big)\tr.
\end{align}
According to the Lipschitz continuity of the truncation mapping $\Gamma^{\tr}_{\Phi,\nu}$  (cf. \cite[ (7.21)]{Li2018})  we arrive at
\begin{align}\la{comp}
\big\v u^{\xi,\tr}(t_{i})- u^{\zeta,\tr}(t_{i})\big\v^2\leq\big\v\br u^{\xi,\tr}(t_{i})-\br u^{\zeta,\tr}(t_{i})\big\v^2.
\end{align}
Making use of inequality $e^{\lambda t_{i+1}}-e^{\lambda t_i}\leq e^{\lambda t_{i+1}}\lambda \tr$ for any $\lambda\in(0,1]$,
we obtain from \eqref{n4.8} and \eqref{comp} that
\begin{align}\la{t4.15}
&e^{\lambda t_{i+1}}\E\v u^{\xi,\tr}(t_{i+1})- u^{\zeta,\tr}(t_{i+1})\v^2
\leq e^{\lambda t_{i+1}}\E\v \br u^{\xi,\tr}(t_{i+1})-\br u^{\zeta,\tr}(t_{i+1})\v^2\nn\\
\leq&e^{\lambda t_{i}}\E\v u^{\xi,\tr}(t_{i})\!-\!u^{\zeta,\tr}(t_{i})\v^2-(b_1
-2K^2\tr^{1-2\nu}-\lambda) e^{\lambda t_{i+1}}\E\v u^{\xi,\tr}(t_{i})-u^{\zeta,\tr}(t_{i})\v^2\tr\nn\\
&+(b_2+2K^2\tr^{1-2\nu})e^{\lambda t_{i+1}}\E\v u^{\xi,\tr}(t_{i-N})-u^{\zeta,\tr}(t_{i-N})\v^2\tr\nn\\
&-b_3e^{\lambda t_{i+1}}\E
V\big(u^{\xi,\tr}(t_{i}),u^{\zeta,\tr}(t_{i})\big)\tr+b_4e^{\lambda t_{i+1}}\E V\big(u^{\xi,\tr}(t_{i-N}),u^{\zeta,\tr}(t_{i-N})\big)\tr.
\end{align}
For any $k\geq0$, summing  \eqref{t4.15} from $i = 0$ to $k$, and together with \eqref{comp} derives
\begin{align}\la{4.16n}
&e^{\lambda t_{k+1}}\E\v u^{\xi,\tr}(t_{k+1})- u^{\zeta,\tr}(t_{k+1})\v^2\nn\\
\leq&\v u^{\xi,\tr}(0)- u^{\zeta,\tr}(0)\v^2
-\big(b_1
-2K^2\tr^{1-2\nu}-\lambda \big)\sum_{i=0}^{k}e^{\lambda t_{i+1}}\E\v u^{\xi,\tr}(t_{i})-u^{\zeta,\tr}(t_{i})\v^2\tr\nn\\
&+(b_2+2K^2\tr^{1-2\nu})\sum_{i=0}^{k}e^{\lambda t_{i+1}}\E\v u^{\xi,\tr}(t_{i-N})-u^{\zeta,\tr}(t_{i-N})\v^2\tr\nn\\
&-b_3\sum_{i=0}^{k}e^{ \lambda t_{i+1}}\E V\big(u^{\xi,\tr}(t_{i}),u^{\zeta,\tr}(t_{i})\big)\tr\nn\\
&+b_4\sum_{i=0}^{k}e^{\lambda t_{i+1}}\E V\big(u^{\xi,\tr}(t_{i-N}),u^{\zeta,\tr}(t_{i-N})\big)\tr\nn\\
\leq &L_{M}-\big(b_1-2K^2\tr^{1-2\nu}-\lambda-(b_2+2K^2\tr^{1-2\nu})e^{\lambda\tau}\big)\nn\\&\times\sum_{i=0}^{k}e^{\lambda t_{i+1}}\E
\v u^{\xi,\tr}(t_{i})-u^{\zeta,\tr}(t_{i})\v^2\tr\nn\\
&-\big(b_3-b_4e^{\lambda\tau}\big)\sum_{i=0}^{k}e^{\lambda t_{i+1}}\E
V\big(u^{\xi,\tr}(t_{i}),u^{\zeta,\tr}(t_{i})\big)\tr.
\end{align}
Making use of \eqref{tr4} and $b_3>b_4$, we choose a $\bar \lambda=\bar \lambda(\hat\tr_2)\in(0,1]$ sufficiently small such that
\begin{align*}
&b_1- 2K^2\hat\tr_2^{1-2\nu}-\bar \lambda-(b_2+2K^2\hat\tr_2^{1-2\nu})e^{\bar \lambda\tau}\geq0~~\hbox{and}~~b_3-b_4e^{\bar \lambda\tau}\geq0.
\end{align*}
Taking $\lambda=\bar\lambda$ in \eqref{4.16n} implies that for any $\tr\in(0,\hat\tr_2]$, $k\geq 0$, and $\xi,\zeta\in B(M)$,
\begin{align*}
\E\v u^{\xi,\tr}(t_{k+1})-u^{\zeta,\tr}(t_{k+1})\v^2\leq L_{M}e^{-\bar \lambda t_{k+1}}.
\end{align*}
The proof is therefore complete.
\end{proof}
Now we formulate the key proposition, which plays an important
role in the analysis of the existence and uniqueness of numerical invariant measures.
\begin{proposition}\la{lemma5}
Suppose that \textup{(H1)--(H3)} hold.
Let $\hat\tr=\hat\tr_1\wedge\hat\tr_2$.
Then for any  $M>0$, $\e>0$, there exists a $\bar T=\bar T(\hat \tr,  M, \e)>\tau$
such that for any $\tr\in(0,\hat\tr]$ and $k\tr\geq \bar T$,
\begin{align}\la{4.24n}
\sup_{\xi,\zeta\in  B(M)}\PP\Big\{\|Y_{t_k}^{\xi,\tr}-Y_{t_k}^{\zeta,\tr}\|\geq\e\Big\}<\e.
\end{align}
\end{proposition}
\begin{proof}
Let $\hat\tr=\hat\tr_1\wedge\hat\tr_2$. For any $M>0$, $\e_1>0$, $\e_2>0$, by virtue of Lemma \ref{le8s} there exist $\tr_4=\tr_4(M, \e_1/6,\e_2/9)\in(0,\hat\tr\wedge\tr_{M}^*]$ and $j^*=j^*(M, \e_1/3,\e_2/3, \tr_4)\geq1$
such that
\begin{align}\la{4.26n}
\sup_{\tr\in(0,\tr_4]}\sup_{k\geq 0}\sup_{\xi\in  B(M)}\PP\Big\{ \sup_{\substack{\v s_1-s_2\v\leq\tau/ j^*\\ s_1,s_2\in[t_{k},t_{k}+\tau]}}\big\v y^{\xi,\tr}(s_1)-y^{\xi,\tr}(s_2)\big\v
\geq\frac{\e_2}{3}\Big\}<\frac{\e_1}{3}.
\end{align}
It follows from  \eqref{n3.7} that for any  $\tr\in(0,\tr_4]$ and $k\geq N$,
\begin{align}\la{le49+1}
&\sup_{\xi,\zeta\in  B(M)}\sup_{t\in[t_{k}-\tau,t_{k}]}\PP\Big\{ \v y^{\xi,\tr}(t)-y^{\zeta,\tr}(t)\v\geq\frac{\e_2}{3}\Big\}\nn\\
\leq&\sup_{\xi,\zeta\in  B(M)}\sup_{i\in\{k-N,\cdots, k-1\}}\sup_{t\in[t_{i},t_{i+1}]}\PP\Big\{\frac{t_{i+1}-t}{\tr}\v u^{\xi,\tr}(t_{i})-u^{\zeta,\tr}(t_{i})\v\geq \frac{\e_2}{6}\Big\}\nn\\
&+\sup_{\xi,\zeta\in  B(M)}\sup_{i\in\{k-N,\cdots, k-1\}}\sup_{t\in[t_{i},t_{i+1}]}\PP\Big\{  \frac{t-t_i}{\tr}\v u^{\xi,\tr}(t_{i+1})-u^{\zeta,\tr}(t_{i+1})\v\geq\frac{\e_2}{6}\Big\}\nn\\
\leq&2\sup_{\xi,\zeta\in  B(M)}\sup_{i\in\{k-N,\cdots, k\}}\PP\Big\{ \v u^{\xi,\tr}(t_{i})-u^{\zeta,\tr}(t_{i})\v\geq \frac{\e_2}{6}\Big\}.
\end{align}
It follows from  the Chebyshev inequality and Lemma \ref{4.4nle} that
\begin{align}\la{le49+2}
&\sup_{i\in\{k-N,\cdots, k\}}\sup_{\xi,\zeta\in  B(M)}\PP\Big\{ \v u^{\xi,\tr}(t_{i})-u^{\zeta,\tr}(t_{i})\v\geq \frac{\e_2}{6}\Big\}\nn\\
\leq&\frac{36}{\e_2^2}\sup_{i\in\{k-N,\cdots, k\}}\sup_{\xi,\zeta\in  B(M)}\E\v u^{\xi,\tr}(t_{i})\!-u^{\zeta,\tr}(t_{i})\v^2
\leq\frac{36 L_{M}e^{-\bar \lambda (t_{k}-\tau)}}{\e_2^2}.
\end{align}
Choose a $T_1=T_1(M, \e_1,\e_2, \tr_4, j^*)>\tau$ sufficiently large such that
\begin{align}\la{le49+3}
\frac{36 L_Me^{-\bar \lambda (T_1-\tau)}}{\e_2^2}<\frac{\e_1}{6j^*}.
\end{align}
Inserting \eqref{le49+2} and \eqref{le49+3} into \eqref{le49+1} implies
\begin{align}\la{4.27n}
\sup_{\tr\in(0,\tr_4]}\sup_{k\tr\geq T_1}\sup_{\xi,\zeta\in  B(M)}\sup_{t\in[t_{k}-\tau,t_{k}]}\PP\Big\{ \v y^{\xi,\tr}(t)-y^{\zeta,\tr}(t)\v\geq\frac{\e_2}{3}\Big\}
<\frac{\e_1}{3j^*}.
\end{align}
According to \eqref{4.26n} and \eqref{4.27n} yields
\begin{align}\la{4.46nnn}
&\sup_{\tr\in(0,\tr_4]}\sup_{k\tr\geq T_1}\sup_{\xi,\zeta\in  B(M)}\PP\Big\{ \|Y_{t_k}^{\xi,\tr}-Y_{t_k}^{\zeta,\tr}\|\geq\e_2\Big\}\nn\\
=&\sup_{\tr\in(0,\tr_4]}\sup_{k\tr\geq T_1}\sup_{\xi,\zeta\in  B(M)}\PP\Big\{\sup_{0\leq j\leq j^*-1}\sup_{t\in[t_{k-N}^{j^*,j},t_{k-N}^{j^*,j+1}]} \v y^{\xi,\tr}(t)-y^{\zeta,\tr}(t)\v\geq\e_2\Big\}\nn\\
\leq&2\sup_{\tr\in(0,\tr_4]}\sup_{k\tr\geq T_1}\sup_{\xi\in  B(M)}\PP\Big\{ \sup_{\substack{\v s_1-s_2\v\leq\tau/j^*\\s_1,s_2\in[t_{k}-\tau,t_{k}]}} \v y^{\xi,\tr}(s_1)-y^{\xi,\tr}(s_2)\v\geq\frac{\e_2}{3}\Big\}\nn\\
&+j^*\sup_{\tr\in(0,\tr_4]}\sup_{k\tr\geq  T_1}\sup_{\xi,\zeta\in B(M)}\sup_{t\in[t_{k}-\tau,t_{k}]}\PP\Big\{ \v y^{\xi,\tr}(t)-y^{\zeta,\tr}(t)\v\geq\frac{\e_2}{3}\Big\}<\e_1.
\end{align}
where $t_k^{j^*,j}=t_k+j\tau/j^*$.
On the other hand, for any $\tr\in[\tr_4,~ \hat\tr]$ and $k\geq N$, it follows from \eqref{n3.7} that
\begin{align*}
&\sup_{\xi,\zeta\in  B(M)}\PP\Big\{ \|Y_{t_k}^{\xi,\tr}-Y_{t_k}^{\zeta,\tr}\|\geq\e_2\Big\}\nn\\
=&\sup_{\xi,\zeta\in  B(M)}\PP\Big\{\sup_{i\in\{k-N,\cdots, k-1\}}\sup_{t\in[t_i,t_{i+1}]} \v y^{\xi,\tr}(t)-y^{\zeta,\tr}(t)\v\geq \e_2\Big\}\nn\\
\leq&\sum_{i=k-N}^{k-1}\sup_{\xi,\zeta\in  B(M)}\PP\Big\{\sup_{t\in[t_i,t_{i+1}]} \frac{t_{i+1}-t}{\tr}\v u^{\xi,\tr}(t_i)-u^{\zeta,\tr}(t_{i})\v
\geq\frac{\e_2}{2}\Big\}\nn\\
&+\sum_{i=k-N}^{k-1}\sup_{\xi,\zeta\in B(M)}\PP\Big\{\sup_{t\in[t_i,t_{i+1}]}\frac{t-t_i}{\tr}\v u^{\xi,\tr}(t_{i+1})-u^{\zeta,\tr}(t_{i+1})\v\geq \frac{\e_2}{2}\Big\}\nn\\
\leq& 2N\sup_{i\geq k-N}\sup_{\xi,\zeta\in  B(M)}\PP\Big\{\v u^{\xi,\tr}(t_i)-u^{\zeta,\tr}(t_{i})\v
\geq\frac{\e_2}{2}\Big\}.
\end{align*}
Choose a $T_2=T_2(M, \e_1,\e_2,\tr_4)>\tau$ sufficiently large such that
$$\frac{8\tau L_Me^{-\bar \lambda (T_2-\tau)}}{\e_2^2\tr_4}<\e_1.$$
Making use of  the Chebyshev inequality and Lemma \ref{4.4nle}  we arrive at
 \begin{align}\la{4.47nnn}
&\sup_{\tr\in[\tr_4, \hat\tr]}\sup_{k\tr\geq T_2}\sup_{\xi,\zeta\in  B(M)}\PP\Big\{ \|Y_{t_k}^{\xi,\tr}-Y_{t_k}^{\zeta,\tr}\|\geq\e_2\Big\}\nn\\
\leq&\frac{8N}{\e_2^2}\sup_{i\geq k-N}\sup_{\xi,\zeta\in B(M)}\E\v u^{\xi,\tr}(t_i)-u^{\zeta,\tr}(t_{i})\v^2
\leq\frac{8\tau L_Me^{-\bar \lambda (T_2-\tau)}}{\e_2^2\tr_4}<\e_1.
\end{align}
Let $\bar T=T_1\vee T_2$ and $\e_1=\e_2=\e$. According to  \eqref{4.46nnn}, and \eqref{4.47nnn} yields
\begin{align*}
&\sup_{\tr\in(0, \hat\tr]}\sup_{k\tr\geq \bar T}\sup_{\xi,\zeta\in B(M)}\PP\big\{\|Y_{t_k}^{\xi,\tr}-Y_{t_k}^{\zeta,\tr}\|\geq\e\big\}\nn\\
\leq&\max\Big\{\sup_{\tr\in(0,\tr_4]}\sup_{k\tr\geq T_1}\sup_{\xi,\zeta\in B(M)}\PP\big\{\|Y_{t_k}^{\xi,\tr}-Y_{t_k}^{\zeta,\tr}\|\geq\e_2\big\},\nn\\
&\sup_{\tr\in[\tr_4, \hat\tr]}\sup_{k\tr\geq T_2}\sup_{\xi,\zeta\in B(M)}\PP\big\{\|Y_{t_k}^{\xi,\tr}-Y_{t_k}^{\zeta,\tr}\|\geq\e_2\big\}\Big\}<\e
\end{align*}
as required.
\end{proof}
\begin{proofs}\emph{Proof of Theorem \textup{\ref{nth3}}}
 Since the proof is rather technical, we divide it into two steps.

\textbf{Step 1}. 
Firstly, choose a special initial data $\xi=\textbf 0$.
For any $\e\in(0,1)$,
in view of \eqref{n4.1} and Proposition \ref{leerry},  there exists a positive constant $\Lambda^*=\Lambda^*(\hat\tr, 0, \e/8)$ such that
\begin{align}\la{th3.1Le46+1}
\sup_{\tr\in(0,\hat\tr]}\sup_{i\geq0}\mu_{t_i}^{\textbf 0,\tr}\big(B^{c}(\Lambda^*)\big)=\sup_{\tr\in(0,\hat\tr]}\sup_{i\geq0}\PP\Big\{\|Y^{\textbf 0,\tr}_{t_i}\|> \Lambda^*\Big\}<\frac{\e}{8}.
\end{align}
By virtue of  Proposition \ref{lemma5},  there exists a $\bar T=\bar T(\hat\tr, \Lambda^*, \e/4)>\tau$ such that
for any $\tr\in(0,\hat\tr]$ and $k\tr\geq \bar T$,
\begin{align}\la{4.32n}
\sup_{X\in  B(\Lambda^*)}\PP\Big\{\|Y^{X,\tr}_{t_k}
-Y^{\textbf 0,\tr}_{t_k}\|\geq\frac{\e}{4}\Big\}<\frac{\e}{4}.
\end{align}
For any $\tr\in(0,\hat \tr]$, $t_k\geq \bar T$, and $i\geq0$, recalling definitions \eqref{n2.1} and \eqref{n4.1}, we obtain
\begin{align*}
&d_{\Xi}\big(\mu_{t_{k+i}}^{\textbf 0,\tr}(\cdot),\mu_{t_{k}}^{\textbf 0,\tr}(\cdot)\big)=\sup_{\Psi\in\Xi}\big\v\E\Psi(Y^{\textbf 0,\tr}_{t_{k+i}})-\E\Psi(Y^{\textbf 0,\tr}_{t_k})\big\v\nn\\
=&\sup_{\Psi\in\Xi}\big\v\E\big(\E\big(\Psi(Y^{\textbf 0,\tr}_{t_{k+i}})
\v\mathcal F_{t_i}\big)\big)-\E\Psi(Y^{\textbf 0,\tr}_{t_k})\big\v\nn\\
=&\sup_{\Psi\in\Xi}\big\v\E\big(\E\Psi(Y^{X,\tr}_{t_k})\v_{X=Y^{\textbf 0,\tr}_{t_i}}\big)-\E\Psi(Y^{\textbf 0,\tr}_{t_k})\big\v\nn\\
\leq&\sup_{\Psi\in\Xi}\int_{C}
\E\big\v\Psi(Y^{X,\tr}_{t_k})-\Psi(Y^{\textbf 0,\tr}_{t_k})\big\v\mu_{t_i}^{\textbf 0,\tr}(\mathrm d X)\nn\\
\leq&\int_{C}\E(2\wedge \|Y^{X,\tr}_{t_k}-Y^{\textbf 0,\tr}_{t_k}\|)\mu_{t_i}^{\textbf 0,\tr}(\mathrm dX).
\end{align*}
Making use of \eqref{th3.1Le46+1} and \eqref{4.32n} yields
\begin{align}\la{n3.48}
d_{\Xi}\big(\mu_{t_{k+i}}^{\textbf 0,\tr}(\cdot),\mu_{t_{k}}^{\textbf 0,\tr}(\cdot)\big)
\leq&\int_{B(\Lambda^*)}\E(2\wedge \|Y^{ X,\tr}_{t_k}-Y^{\textbf 0,\tr}_{t_k}\|)\mu_{t_i}^{\textbf 0,\tr}(\mathrm d X)+\frac{\e}{4}\nn\\
\leq&2\int_{B(\Lambda^*)}\PP\big\{\|Y^{ X,\tr}_{t_k}
-Y^{\textbf 0,\tr}_{t_k}\|\geq\frac{\e}{4}\big\}\mu_{t_i}^{\textbf 0,\tr}(\mathrm dX)+\frac{\e}{4}+\frac{\e}{4}\nn\\
\leq&2\sup_{X\in  B(\Lambda^*)}\PP\Big\{\|Y^{X,\tr}_{t_k}
-Y^{\textbf 0,\tr}_{t_k}\|\geq\frac{\e}{4}\Big\}+\frac{\e}{4}+\frac{\e}{4}<\e.
\end{align}
This implies that the measure sequence $\{\mu_{t_k}^{\textbf 0,\tr}(\cdot)\}_{k\geq0}$ is uniformly Cauchy.
Since $(\mathcal P(C), d_{\Xi})$  is complete (see \cite[Corollary 10.5]{Dudley}),
there exists a unique probability measure $\pi^{\tr}(\cdot)$
  such that
\begin{align}\la{4.34n}
\lim_{t_k\rightarrow \infty}d_{\Xi}\big(\mu_{t_k}^{\textbf 0,\tr}(\cdot),\pi^{\tr}(\cdot)\big)=0,~~~\hbox{uniformly~in~} \tr\in(0,\hat\tr].
\end{align}
\textbf {Step 2}. 
Let $\xi\in\mathcal C^{\alpha}_{\mathcal F_0}$.
It is
straightforward to  see that
$
\|\xi\|<\infty,~ \PP-\hbox{a.s}.
$
Hence for any $\e>0$ there exists an $M>0$ sufficiently large such that
	\begin{align}\la{set}
		P_{\xi}(B^c(M))=\PP\Big\{ \|\xi\|>M\Big\}
		<\frac{\e}{8},
	\end{align}
where
\begin{align*}
P_{\xi}(A):=\PP\Big\{\omega\in\Omega: \xi(\omega)\in A\Big\}, ~~~\forall~ A\in\mathfrak B(C).
\end{align*}
By  Proposition \ref{lemma5} there exists a $\bar T_1=\bar T_1(\hat\tr, M, \e/4)>\tau$   sufficiently large such that  for any $\tr\in(0,\hat\tr]$ and $k\tr\geq\bar T_1$,
\begin{align}\la{4.28n}
\sup_{X\in B(M)}\PP\Big\{\|Y^{X,\tr}_{t_k}-Y^{\textbf 0,\tr}_{t_k}\|\geq\frac{\e}{4}\Big\}
<\frac{\e}{4}.
\end{align}
For any  $\tr\in(0,\hat\tr]$ and $k\tr\geq\bar T_1$,
using \eqref{set} and \eqref{4.28n} implies
\begin{align*}
&d_{\Xi}\big(\mu_{t_k}^{\xi,\tr}(\cdot),\mu_{t_k}^{\textbf 0,\tr}(\cdot)\big)
\leq\E\Big(\E\big(2\wedge \|Y^{X,\tr}_{t_k}-Y^{\textbf 0,\tr}_{t_k}\|\big)\big\v_{\{X=\xi\}}\Big)\nn\\
=&\int_{ B(M)}\E\big(2\wedge \|Y^{X,\tr}_{t_k}-Y^{\textbf 0,\tr}_{t_k}\|\big)P_{\xi}(\mathrm dX)\nn\\&+\int_{B^c(M)}\E\big(2\wedge \|Y^{X,\tr}_{t_k}-Y^{\textbf 0,\tr}_{t_k}\|\big)P_{\xi}(\mathrm dX)\nn\\
\leq&2\int_{ B(M)}\PP\big\{\|Y^{X,\tr}_{t_k}-Y^{\textbf 0,\tr}_{t_k}\|
\geq\frac{\e}{4}\big\}P_{\xi}(\mathrm dX)+\frac{\e}{4}+2P_{\xi}(B^c(M))<\e.
\end{align*}
Therefore, for any $\xi\in \mathcal C_{\mathcal F_0}^{\alpha}$,
\begin{align}\la{4.29n}
\lim_{t_k\rightarrow \infty}d_{\Xi}\big(\mu_{t_k}^{\xi,\tr}(\cdot),\mu_{t_k}^{\textbf 0,\tr}(\cdot)\big)=0,~~~\hbox{uniformly ~in~}\tr\in(0,\hat\tr].
\end{align}
It follows from \eqref{4.34n} and \eqref{4.29n} that
\begin{align*}
&\lim_{t_k\rightarrow \infty}d_{\Xi}(\mu_{t_k}^{\xi,\tr}(\cdot),\pi^{\tr}(\cdot))\nn\\
\leq& \lim_{t_k\rightarrow \infty}d_{\Xi}(\mu_{t_k}^{\textbf 0,\tr}(\cdot),\pi^{\tr}(\cdot))+\lim_{t_k\rightarrow \infty}\big(\mu_{t_k}^{\xi,\tr}(\cdot),\mu_{t_k}^{\textbf 0,\tr}(\cdot)\big)\nn\\
=&0,~~~\hbox{uniformly ~in~}\tr\in(0,\hat\tr].
\end{align*}
The required assertion \eqref{th3+1} follows.
 By the similar way as Step $2$, for any $M>0$, we may also prove that the convergence in \eqref{th3+1} is also uniform for the initial data $\xi\in \mathbb{B}( M,\alpha)$.
The proof is complete.
\end{proofs}

Next, we give the convergence  between the numerical segment process $Y^{\textbf 0,\tr}_{t_k}$ and the exact one $ x^{\textbf 0}_{t_k} $.
\begin{lemma}\la{noth2}
Suppose  that $(\textup{H}1)$ and $(\textup{H}2)$ hold.
Then
for any  $\e>0$, $T>0$,  there exists a $\tr_5=\tr_5(\e,T)\in(0,\hat\tr_1]$ such that 
\begin{align*}
\sup_{\tr\in(0,\tr_5]}\sup_{0\leq k\tr\leq T}\PP\Big\{\|Y^{{\textbf{\textup 0}},\tr}_{t_k}-x^{{\textbf{\textup 0}}}_{t_{k}}\|\geq\e\Big\}<\e.
\end{align*}
\end{lemma}
\begin{proof}\textbf{Proof.}
Without loss of generality, for any  $\e>0$,  $T>\tau$,
by Lemma \ref{nle3} there is a $\tr_3=\tr_3(0, \e/4,T)\in(0,\hat\tr_1]$ such that
\begin{align}\la{n3.30}
\sup_{\tr\in(0,\tr_3]}\PP\Big\{\beta_{\tr_3,-N}^{{\textbf 0},\tr}<T\Big\}<\frac{\e}{4},
\end{align}
where $\beta_{\tr_3,-N}^{{\textbf 0},\tr}$ is given by \eqref{3.20n}.
In view of Lemma \ref{nole4},  choose a $\tr_4=\tr_4(0, \e/8, \e/2)\in(0,\tr_3]$ sufficiently small  such that
\begin{align}\la{n3.31}
\sup_{\tr\in(0,\tr_4]}\sup_{k\geq0}\PP\Big\{\sup_{t\in[t_k,t_{k}+\tau]}\v y^{{\textbf 0},\tr}(t)-z^{{\textbf 0},\tr}(t)\v \geq\frac{\e}{2}\Big\}<\frac{\e}{8}.
\end{align}
For any positive constant $\ell$, define
\begin{align*}
\delta_{\ell}^{{\textbf 0}}=\inf\Big\{t\geq-\tau: \v x^{{\textbf 0}}(t)\v> \ell\Big\}.
\end{align*}
By virtue of \cite[Theorem 2.1]{Song-Li2021} we have
\begin{align*}
\PP\Big\{\delta_{\ell}^{{\textbf 0}}< T\Big\}\leq\frac{L}{\ell^2},
\end{align*}
Choose an $\ell$ sufficiently large such that
\begin{align}\la{delta}
\PP\Big\{\delta_{\ell}^{{\textbf 0}}< T\Big\}<\frac{\e}{4}.
\end{align}
Let $\gamma^{\tr_3,\tr}_{{\textbf 0},\ell}:=\delta_{\ell}^{{\textbf 0}}\wedge \beta_{\tr_3,-N}^{{\textbf 0},\tr}$.
By the similar way as \cite[Theorem 3.3]{Song-Li2021},
 there exists a $\tr_5=\tr_5(\e,T)\in(0,\tr_4]$ such that
\begin{align*}
\sup_{\tr\in(0,\tr_5]}\E\Big(\sup_{0\leq t\leq T}\v z^{{\textbf 0},\tr}(t)-x^{{\textbf 0}}(t)\v^2\textbf{1}_ {\{\gamma^{\tr_3,\tr}_{{\textbf 0},\ell}\geq T\}}\Big)<\frac{\e^3}{16}.
\end{align*}
This, along with the Chebyshev inequality, implies that for any $\tr\in(0,\tr_5]$,
\begin{align}\la{n3.32}
&\sup_{\tau\leq k\tr\leq T}\PP\Big\{\gamma^{\tr_3,\tr}_{{\textbf 0},\ell}\geq T, \sup_{t\in[t_k-\tau,t_k]}\v z^{{\textbf 0},\tr}(t)-x^{{\textbf 0}}(t)\v\geq\frac{\e}{2}\Big\}\nn\\
\leq&\PP\Big\{\gamma^{\tr_3,\tr}_{{\textbf 0},\ell}\geq T, \sup_{0\leq t\leq T}\v z^{{\textbf 0},\tr}(t)-x^{{\textbf 0}}(t)\v\geq\frac{\e}{2}\Big\}\nn\\
\leq&\frac{4}{\e^2}\E\big(\sup_{0\leq t\leq T}\v z^{{\textbf 0},\tr}(t)-x^{{\textbf 0}}(t)\v^2\textbf{1}_ {\{\gamma^{\tr_3,\tr}_{{\textbf 0},\ell}\geq T\}}\big)
<\frac{\e}{4}.
\end{align}
It follows from \eqref{n3.31} and \eqref{n3.32} that
\begin{align}\la{n3.33}
&\sup_{\tr\in(0,\tr_5]}\sup_{\tau\leq k\tr\leq T}\PP\Big\{\gamma^{\tr_3,\tr}_{{\textbf 0},\ell}\geq T, \|Y^{{\textbf 0},\tr}_{t_k}-x^{{\textbf 0}}_{t_{k}}\|\geq\e\Big\}\nn\\
\leq&\sup_{\tr\in(0,\tr_5]}\sup_{\tau\leq k\tr\leq T}\PP\Big\{ \sup_{t\in[t_k-\tau,t_k]}\v y^{{\textbf 0},\tr}(t)-z^{{\textbf 0},\tr}(t)\v\geq\frac{\e}{2}\Big\}\nn\\
&+\sup_{\tr\in(0,\tr_5]}\sup_{\tau\leq k\tr\leq T}\PP\Big\{\gamma^{\tr_3,\tr}_{{\textbf 0},\ell}\geq T, \sup_{t\in[t_k-\tau,t_k]}\v z^{{\textbf 0},\tr}(t)-x^{{\textbf 0}}(t)\v\geq\frac{\e}{2}\Big\}
<\frac{3\e}{8}.
\end{align}
Since for any $\tr\in(0,\tr_5]$, $$y^{{\textbf 0},\tr}(\theta)=\Gamma_{\Phi,\nu}^{\triangle}(0)=0=x^{{\textbf 0},\tr}(\theta),~~~\theta\in[-\tau,0],$$
it is obvious that
\begin{align*}
&\sup_{0\leq k\tr\leq \tau}\|Y_{t_k}^{{\textbf 0},\tr}-x^{{\textbf 0},\tr}_{t_k}\|=\sup_{-\tau\leq t\leq \tau}\v y^{{\textbf 0},\tr}(t)-x^{{\textbf 0},\tr}(t)\v\nn\\
=&\sup_{0\leq t\leq \tau}\v y^{{\textbf 0},\tr}(t)-x^{{\textbf 0},\tr}(t)\v=\|Y_{\tau}^{{\textbf 0},\tr}-x^{{\textbf 0},\tr}_{\tau}\|.
\end{align*}
According to \eqref{n3.30}, \eqref{delta},  and \eqref{n3.33} yields
\begin{align}\la{4.59nnn}
&\sup_{\tr\in(0,\tr_5]}\sup_{0\leq k\tr\leq T}\PP\Big\{\|Y^{{\textbf 0},\tr}_{t_k}-x^{{\textbf 0}}_{t_{k}}\|\geq\e\Big\}\nn\\
=&\sup_{\tr\in(0,\tr_5]}\sup_{\tau\leq k\tr\leq T}\PP\Big\{\|Y^{{\textbf 0},\tr}_{t_k}-x^{{\textbf 0}}_{t_{k}}\|\geq\e\Big\}\nn\\
\leq&\sup_{\tr\in(0,\tr_5]}\sup_{\tau\leq k\tr\leq T}\PP\Big\{\gamma^{\tr_3,\tr}_{\textbf 0,\ell}\geq T, \|Y^{{\textbf 0},\tr}_{t_k}-x^{{\textbf 0}}_{t_{k}}\|\geq\e\Big\}\nn\\
&+\sup_{\tr\in(0,\tr_5]}\PP\Big\{\beta_{\tr_{3},-N}^{{\textbf 0},\tr}< T\Big\}
+\PP\Big\{\delta_{\ell}^{{\textbf 0}}< T\Big\}
<\e.
\end{align}
The proof is
therefore complete.
\end{proof}

\begin{proofs}\emph{Proof of Theorem \textup{\ref{nth4}}}
For any $\e\in(0,1)$, in view of Lemma \ref{nth1}  there exists a  $ T_1^{*}>0$ such that for any $t\geq T_1^*$,
\begin{align}\la{3.47n}
d_{\Xi}\big(\pi(\cdot),\mu_{t}^{{\textbf 0}}(\cdot)\big)<\frac{\e}{3}.
\end{align}
By virtue of Theorem \ref{nth3}  there exists a $T_2^*>\tau$  such that for any $\tr\in(0,\hat \tr]$ and $k\tr\geq T_2^*$,
\begin{align}\la{3.48n}
d_{\Xi}\big(\mu_{t_k}^{{\textbf 0},\tr}(\cdot),\pi^{\tr}(\cdot)\big)<\frac{\e}{3}.
\end{align}
Let $T^*:=T_1^*\vee T_2^*$.
By Lemma \ref{noth2}  there exists a $\tr^*\in(0,\hat\tr]$ such that for any $\tr\in(0,\tr^*]$ and $0\leq k\tr\leq T^*+1$,
\begin{align}\la{3.49n}
\PP\Big\{\|x^{{\textbf 0}}_{t_k}-Y^{{\textbf 0},\tr}_{t_k}\|\geq\frac{\e}{6}\Big\}<\frac{\e}{12}.
\end{align}
Furthermore, due to the definition of $d_{\Xi}$, and \eqref{3.49n}, we deduce that for any $\tr\in(0,\tr^*]$ and $0\leq k\tr\leq T^*+1$,
\begin{align}\la{3.50n}
d_{\Xi}\big(\mu_{t_k}^{{\textbf 0}}(\cdot),\mu_{t_k}^{{\textbf 0},\tr}(\cdot)\big)
=& \sup_{\Psi\in\Xi}\E\big\v\Psi(x^{{\textbf 0}}_{t_k})-\Psi(Y^{{\textbf 0},\tr}_{t_k})\big\v
\leq \E\big(2\wedge\|x^{{\textbf 0}}_{t_k}-Y^{{\textbf 0},\tr}_{t_k}\|\big)\nn\\
=&\E\big((2\wedge\|x^{{\textbf 0}}_{t_k}-Y^{{\textbf 0},\tr}_{t_k}\|)\textbf 1_{\{\|x^{{\textbf 0}}_{t_k}-Y^{{\textbf 0},\tr}_{t_k}\|\geq\frac{\e}{6}\}}\big)\nn\\
&+ \E\big((2\wedge\|x^{{\textbf 0},\tr}_{t_k}-Y^{{\textbf 0},\tr}_{t_k}\|)\textbf 1_{\{\|x^{{\textbf 0}}_{t_k}-Y^{{\textbf 0},\tr}_{t_k}\|<\frac{\e}{6}\}}\big)\nn\\
\leq &\frac{\e}{6}+\frac{\e}{6}=\frac{\e}{3}.
\end{align}
For any  $\tr\in(0,\tr^*]$, choose an integer $k_0$ such that $T^*\leq k_0\tr\leq T^*+1$.
It follows from \eqref{3.47n}, \eqref{3.48n}, and \eqref{3.50n} that
\begin{align}
d_{\Xi}(\pi(\cdot),\pi^{\tr}(\cdot))
\leq&d_{\Xi}\big(\pi(\cdot),\mu_{t_{k_0}}^{{\textbf 0}}(\cdot)\big)+d_{\Xi}(\mu_{t_{k_0}}^{{\textbf 0},\tr}(\cdot),\pi^{\tr}(\cdot))+d_{\Xi}(\mu_{t_{k_0}}^{{\textbf 0}}(\cdot),\mu_{t_{k_0}}^{{\textbf 0},\tr}(\cdot))\nn\\
<&\frac{\e}{3}+\frac{\e}{3}+\frac{\e}{3}=\e.
\end{align}
The proof is therefore complete.
\end{proofs}
\section{Numerical experiments}\la{exper}
In this section we provide an example and numerical simulations to illustrate the efficiency of TEMSP \eqref{temlis}.
\begin{example}
{\rm
Consider the following  nonlinear SDDE
\begin{align}\la{n5.1}
\left\{
\begin{array}{ll}
\mathrm dx_{1}(t)=(1-x_1(t)-3x_{1}^3(t))\mathrm dt+x^2_{2}(t-1)\mathrm dW_1(t)\\
\mathrm dx_{2}(t)=-(x_2(t)+3x_{2}^3(t))\mathrm dt+x^2_{1}(t-1)\mathrm dW_2(t)
\end{array}
\right.
\end{align}
with  different initial data \begin{align*}
&\xi_1(\theta)=(\xi_{11}(\theta),\xi_{12}(\theta))^{T}=(B_1(-\theta),B_2(-\theta))^T,\nn\\
&\xi_2(\theta)=(\xi_{21}(\theta),\xi_{22}(\theta))^{T}=(2\theta, \theta+1)^T,\nn\\
&\xi_3(\theta)=(\xi_{31}(\theta),\xi_{32}(\theta))^{T}=(-3,4)^T
\end{align*}
for any $\theta\in[-\tau,0]$, where
$(B_1(\cdot), B_2(\cdot))$ is a two-dimensional Brownian motion, which is independent of $(W_1(\cdot),W_2(\cdot))$.

For any  $R>0$ and $x,~\bar x,~y,~\bar y\in \RR$ with $\v x\v\vee\v \bar x\v\vee\v y\v \vee\v\bar y\v\leq R$, we compute
\begin{align}\la{n5.2}
&\v f(x, y)-f(\bar{x},\bar  y)\v\nn\\
=&\big\v(1+3x_1^2+3x_1\bar x_1+3\bar x_1^2)^2(x_1-\bar x_1)^2+(1+3x_2^2+3x_2\bar x_2+3\bar x_2^2)^2(x_2-\bar x_2)^2\big\v^{1/2}\nn\\
\leq&\big\v(1+9R^2)^2(x_1-\bar x_1)^2+(1+9R^2)^2(x_2-\bar x_2)^2\big\v^{1/2}=(1+9R^2)\v x-\bar x\v,
\end{align}
and
\begin{align}\la{5.2n}
\v g(x, y)-g(\bar{x},\bar  y)\v
=&\big\v(y_2+\bar y_2)^2(y_2-\bar y_2)^2+(y_1+\bar y_1)^2(y_1-\bar y_1)^2\big\v^{1/2}\nn\\
\leq&2R\v y-\bar y\v\leq(2R\vee 4R^2)(1\wedge\v y-\bar y\v).
\end{align}
It is straightforward to see that (H1) holds.
In addition, it is easy to verify that for any $x,~\bar x,~y,~\bar y\in \RR$,
\begin{align*}
&\big\langle2x ,f(x,y)\big\rangle+\v g(x,y)\v^2\nn\\
= &2x_1-2\v x\v^2-6(x_1^4+x_2^4)+(y_1^4+y_2^4)
\leq L-3\v x\v^4+\v y\v^4,
\end{align*}
which implies that (H2) holds with $a_2=3,a_3=1$. Furthermore,
\begin{align*}
&2\big\langle x-\bar x ,f(x,y)-f(\bar x,\bar y)\big\rangle+\v g(x,y)-g(\bar x,\bar y)\v^2\nn\\
=&-2\v x-\bar x\v^2-6\big(x_1^2+x_1\bar x_1+\bar x_1^2\big)\v x_1-\bar x_1\v^2-6\big(x_2^2+x_2\bar x_2+\bar x_2^2\big)\v x_2-\bar x_2\v^2\nn\\
&+(y_1+\bar y_1)^2(y_1-\bar y_1)^2+(y_2+\bar y_2)^2(y_2-\bar y_2)^2\nn\\
\leq&-2\v x-\bar x\v^2-3(x_1+\bar x_1)^2(x_1-\bar x_1)^2-3(x_2+\bar x_2)^2(x_2-\bar x_2)^2\nn\\
&+(y_1+\bar y_1)^2(y_1-\bar y_1)^2+(y_2+\bar y_2)^2(y_2-\bar y_2)^2,
\end{align*}
which implies that (H3) holds with $b_1=2, b_2=0$.  Remark \ref{re1} allows us to conclude that
the segment process $\{x_{t}^{\xi}\}_{t\geq 0}$ of \eqref{n5.1} has a unique invariant measure $\pi(\cdot)\in\mathcal{P}(C)$.

According to \eqref{n3.1}, \eqref{n5.2} and \eqref{5.2n}, we take  $\Phi(R)=16R^4$ for all $R\geq1$. Then, $$\Phi^{-1}(R)=R^{1/4}/2,~~~\hbox{for~}R\geq 16.$$
 Let $\nu=1/100$. By virtue of \eqref{n3.2} a direct
computation yields that for any $\tr\in(0,1)$,
$$\Gamma_{\Phi,\nu}^{\tr}(x)=\Big(\v x\v\wedge\tr^{-\frac{1}{400}}\Big)\frac{x}{\v x\v}.$$
Choose $\hat\tr=10^{-3}$, and compute
\begin{align*}
&1.1264=a_2-6K^2 \hat\tr^{1-2\nu}>a_3=1~~\hbox{and}~~ 0.8243=b_1- 4K^2\hat\tr^{1-2\nu}>b_2=0.
\end{align*}
This, along with Theorem \ref{nth3} and  Theorem \ref{nth4} implies that for any $\tr\in(0,\hat\tr]$,  TEMLISP $\{
Y_{t_k}^{\xi,\tr}\}_{k\geq 0}$ defined by \eqref{temlis} is asymptotically stable in distribution and admits a unique numerical invariant measure $\pi^{\tr}(\cdot)$ satisfying $\lim_{\tr\rightarrow0}d_{\Xi}(\pi(\cdot),\pi^{\tr}(\cdot))=0$.

To test the efficiency of  TEMSP \eqref{temlis},  we carry out some numerical simulations  using MATLAB.  
For each of the numerical experiments performed,  the red dotted line,  the blue long and short dash line, and the green line represent the sample means of $\{\Psi_l(Y_{t_k}^{\xi_i,\tr_j})\}_{k\geq 0}$ by TEMSP  starting from different initial data $\xi_1,~\xi_2,~\xi_3$, respectively.
Figure \ref{mean} depicts  the sample means of $\Psi_l(Y^{\xi_i,\tr_j}_{t_k})$ ($j,l=1,2$)  with different initial data $\xi_{i}~(i=1,2,3)$ in the interval $[0,10]$ for $2000$ sample points and different step sizes $\tr_1= 10^{-3},~\tr_2=10^{-4} $,   where test functionals  $\Psi_1(\cdot)=\cos(\|\cdot\|)$ and $\Psi_2(\cdot)=2\wedge\|\cdot \|$.
Figure \ref{mean} depicts that each $\E\Psi_l(Y^{\xi_i,\tr_j}_{t_k})$  starting from different initial data tends to a constant as $t_k\rightarrow\infty$, which implies the existence of numerical invariant measure.
Figure \ref{ECDF} displays  empirical cumulative
distribution functions (Empirical CDF) of   $\Psi_l(Y^{\xi_i,\tr_j}_{10})$ for $2000$ sample points.
 \begin{figure}[htp]
  \begin{center}
\includegraphics[width=12cm,height=8cm]{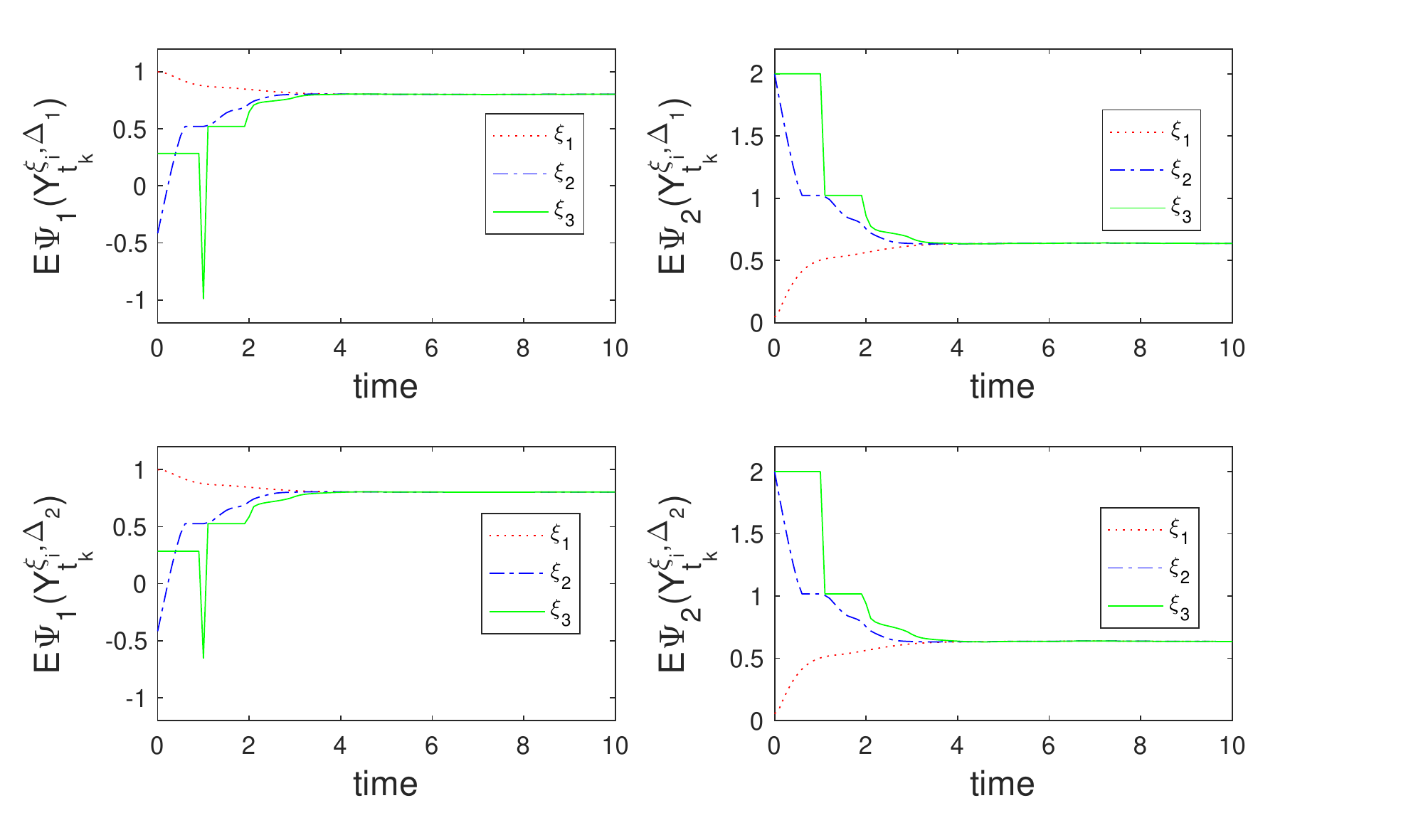}
   \end{center}
 \caption{Sample means of   $\Psi_l(Y^{\xi_i,\tr_j}_{t_k})$ $(i=1,2,3;j,l=1,2)$ 
 with different initial data $\xi_{i} $ in the interval $[0,10]$ for 2000 sample points and different step sizes   $\tr_1=10^{-3},~\tr_2=10^{-4}$, and test functionals   $\Psi_1(\cdot)=\cos(\|\cdot\|)$, $\Psi_2(\cdot)=2\wedge\|\cdot\|$.}
\label{mean}
\end{figure}
 \begin{figure}[htp]
  \begin{center}
\includegraphics[width=12cm,height=8cm]{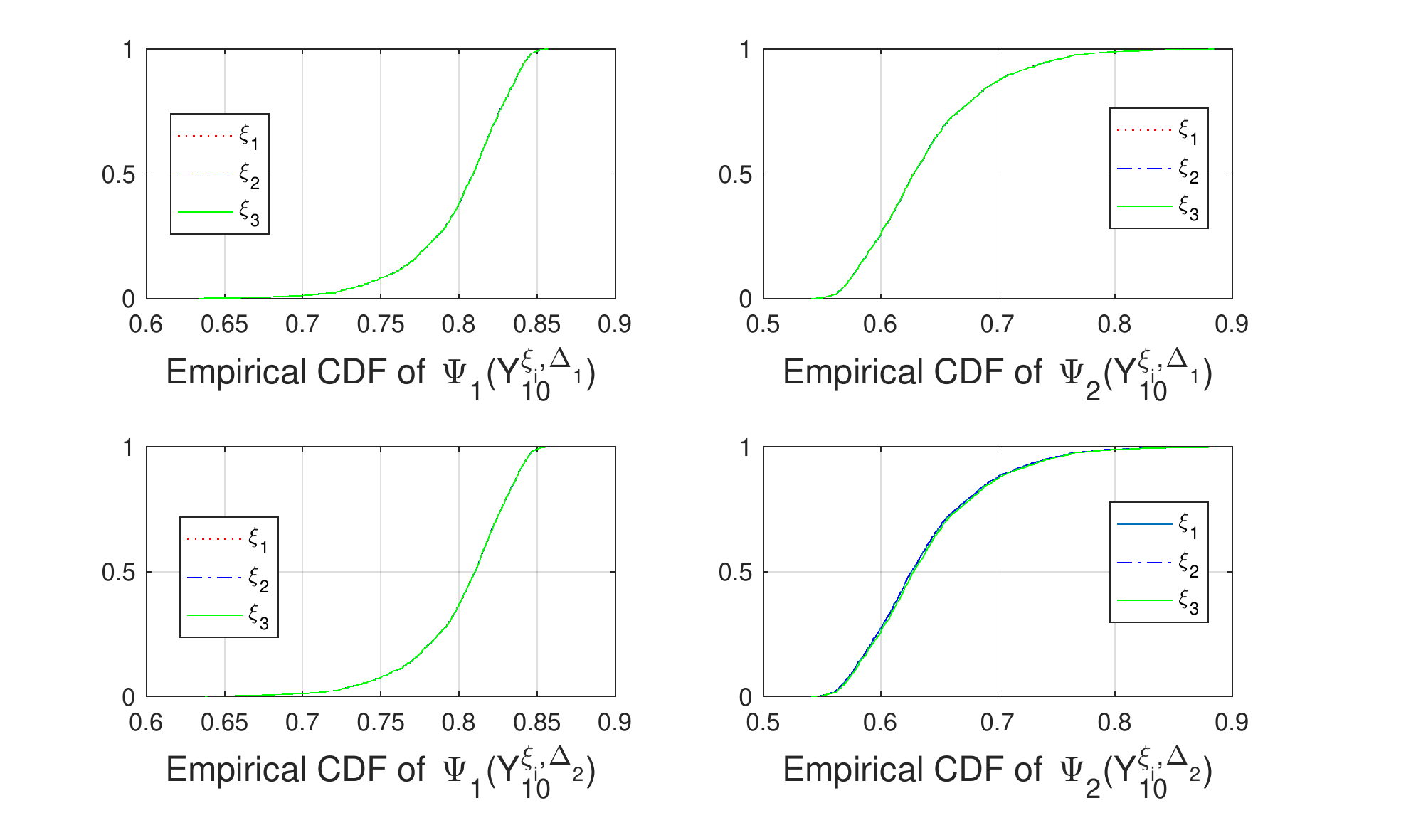}
   \end{center}
 \caption{Empirical cumulative
distribution functions of   $\Psi_l(Y^{\xi_i,\tr_j}_{10})$ $(j,l=1,2)$ with different initial data $\xi_{i}~(i=1,2,3)$  for 2000 sample points and different step sizes $\tr_{j}~(j=1,2)$, where step sizes $\tr_1=10^{-3},~\tr_2=10^{-4}$, and  test functionals  $\Psi_1(\cdot)=\cos(\|\cdot\|)$, $\Psi_2(\cdot)=2\wedge\|\cdot\|$.}
\label{ECDF}
\end{figure}
}
\end{example}

\section{Summary}\la{summ}
We investigate the explicit approximation of the invariant measure for nonlinear SDDEs with non-globally Lipschitz diffusion coefficients. The appropriate numerical segment processes  TEMSP are proposed. Since the mean square of the  exact solutions  may be not  uniformly bounded and attracted, to overcome this difficulty,  we take advantage of  the linear structure of  TEMSP  to prove the uniform boundedness and attraction  in probability. Finally we yield the existence of the unique numerical invariant measure, which converges to the exact one in the  Fortet-Mourier distance. 

\backmatter

\section*{Declarations}
\begin{itemize}
\item\textbf{Ethics approval} ~Not applicable.
\item\textbf{Availability of supporting data}~All data or codes generated during the study are available from the corresponding author by request.
\item \textbf{Conflict of interest/Competing interests}~The authors declare no competing interests.
\item \textbf{Funding}
~The research of Xiaoyue Li was supported by the National Natural Science Foundation of China (No. 11971096), the National Key R$\&$D Program of China (2020YFA0714102), the Natural Science Foundation of Jilin Province (No. YDZJ202101ZYTS154), the Education Department of Jilin Province (No. JJKH20211272KJ), and the Fundamental Research Funds for the Central Universities. And the research of Xuerong Mao was supported by the Royal Society (WM160014, Royal Society Wolfson Research Merit Award),
the Royal Society of Edinburgh (RSE1832),
and Shanghai Administration of Foreign Experts Affairs (21WZ2503700, the Foreign Expert Program).
\item \textbf{Authors' contributions}~Authors in this paper contributed equally to this work.
\item\textbf{Acknowledgments} ~Not applicable.
\end{itemize}


\begin{thebibliography}{99}

\bibitem{BM04a}  A. Bahar, X. Mao:
Stochastic delay Lotka-Volterra model.
{J. Math. Anal. Appl.,} {\bf 292(2)}, 364--380, (2004).

\bibitem{BM04b}
A. Bahar, X. Mao:
Stochastic delay population dynamics.
{Int. J. Pure Appl. Math.,} {\bf 11(4)},
377--400, (2004).

\bibitem{bao-shao-yuan}
J. Bao, J. Shao, C. Yuan:
Invariant measures for path--dependent
random diffusions. {arXiv}:1706.05638, (2017).
\bibitem{bao-yin-yuan}
J. Bao, G. Yin, C. Yuan:
Ergodicity for functional stochastic differential equations
and applications.
{Nonlinear Anal.,} {\bf 98}, 66--82, (2014).
\bibitem{Basak-Bhattacharya}
G. K. Basak, R. N. Bhattacharya:
Stability in distribution for a class of singular diffusions.
{Ann. Probab.,} {\bf 20(1)},  312--321, (1992).
\bibitem{Butkovsky}
O. Butkovsky:
Subgeometric rates of convergence of Markov processes in the Wasserstein metric.
{Ann. Appl. Probab.,} {\bf 24(2)}, 526--552, (2014).
\bibitem{Cloez-Hairer}
B. Cloez, M. Hairer:
Exponential ergodicity for Markov processes
with random switching.
{Bernoulli,} {\bf 21(1)}, 505--536, (2015).
\bibitem{C-H-S}
J. Cui,  J. Hong, L. Sun: Weak convergence and invariant measure of a full discretization for parabolic SPDEs with non-globally Lipschitz coefficients.
{Stochastic Process. Appl.,} {\bf134}, 55--93, (2021).

\bibitem{DLMP08}
F. Deng, Q. Luo, X. Mao, S. Pang:
Noise suppresses or expresses exponential growth.
{Systems Control Lett.,}  {\bf 57}, 262--270, (2008).


\bibitem{Dudley}
R. M. Dudley:
{Probabilities and Metrics. Convergence of Laws on Metric Spaces, with a View to Statistical Testing.}
Lecture Notes Series, (1976).
\bibitem{Fang-Giles}
W. Fang, M. B. Giles:
Adaptive Euler-Maruyama method for SDEs with nonglobally Lipschitz drift.
{Ann. Appl. Probab.,} {\bf 30(2)}, 526--560, (2020).
\bibitem{Hairer2005}
M. Hairer:
Ergodicity of stochastic differential equations driven by fractional Brownian motion.
{Ann. Probab.,} {\bf 33(2)}, 703-758, (2005).

\bibitem{Hairer2008}
M. Hairer: Ergodic theory for stochastic PDEs. 
www.hairer.org/notes/Imperial.pdf, (2008).

\bibitem{Hairer-Mattingly-Scheutzow}
M. Hairer, J. C. Mattingly, M. Scheutzow:
Asymptotic coupling and a general form of Harris'
theorem with applications to stochastic delay equations.
{Probab. Theory Related Fields,} {\bf 149 (1--2)}, 223--259, (2011).
\bibitem{hairer-ohashi}
M. Hairer, A. Ohashi:
Ergodic theory for SDEs with extrinsic memory.
{Ann. Probab.,} {\bf 35(5)}, 1950--1977, (2007).

\bibitem{HN18}
A. Hening, D. H. Nguyen:
Coexistence and extinction
for stochastic Kolmogorov systems.
{Ann. Appl. Probab.,} {\bf 28(3)}, 1893--1942, (2018).

\bibitem{H-W}
 J. Hong,  X. Wang: Invariant Measures for Stochastic Nonlinear Schr\"{o}dinger Equations, Numerical Approximations and Symplectic Structures. Springer, (2019).
 \bibitem{H-W-Z}
J. Hong, X. Wang, L. Zhang: Numerical analysis on ergodic limit of approximations for stochastic NLS equation via multi-symplectic scheme.
{SIAM J. Numer. Anal.,} {\bf 55(1)}, 305--327, (2017).

\bibitem{HM10}
L. Huang, X. Mao:
SMC design for robust $H_\infty$ control of uncertain stochastic delay systems.
 {Automatica,} {\bf 46(2)}, 405--412, (2010).

\bibitem{Khasminskii}
R. Khasminskii:
{Stochastic Stability of Differential Equations}.
Second edition, Springer, (2012).
\bibitem{Kolmanovskii-Nosov}
V. B. Kolmanovskii, V. R. Nosov:
{Stability of Functional Differential Equations.}
Academic Press, Inc., (1986).
\bibitem{Kulik-Scheutzow}
 A. Kulik, M. Scheutzow:
Well-posedness, stability and sensitivities for stochastic delay equations: a generalized coupling approach.
{Ann. Probab.,} {\bf48(6)}, 3041--3076, (2020).
\bibitem{Lamba-Mattingly-Sstuart}
H. Lamba, J. C. Mattingly, A. M. Stuart:
An adaptive Euler--Maruyama scheme for SDEs:
convergence and stability.
{IMA J. Numer. Anal.,} {\bf 27(3)}, 479--506, (2007).
\bibitem{Li-Ma-Yang-Yuan}
X. Li, Q. Ma, H. Yang, C. Yuan:
The numerical invariant measure of stochastic differential equations with Markovian switching.
{SIAM J. Numer. Anal.,} {\bf 56(3)}, 1435--1455, (2018).
\bibitem{Li2018}
X. Li, X. Mao, G. Yin:
Explicit numerical approximations for stochastic differential equations in finite
and infinite horizons: truncation methods, convergence in $p$th moment and stability.
{IMA J. Numer. Anal.,} {\bf 39(2)}, 847--892, (2019).
\bibitem{Mao2007}
X. Mao:
{Stochastic Differential Equations and Applications}.
Second edition, Horwood Publishing Limited,  (2007).
\bibitem{M11}
X. Mao: Stationary distribution of stochastic
population systems.
{Systems Control Lett.,} {\bf 60(6)}, 398--406, (2011).

\bibitem{Mao2005}
X. Mao,  M. J. Rassias:
Khasminskii-type theorems for stochastic differential delay equations.
\emph{Stoch. Anal. Appl.,} {\bf 23(5)}, 1045--1069, (2005).

\bibitem{Yuan-Mao-Hybrid}
X. Mao, C. Yuan:
{Stochastic Differential Equations with Markovian Switching.}  Imperial College Press, (2006).
\bibitem{Mattingly-Staurt-Higham}
J. C. Mattingly, A. M. Stuart, D. J. Higham:
Ergodicity for SDEs and approximations: locally Lipschitz vector fields and degenerate noise.
\emph{Stochastic Process. Appl.,} {\bf 101(2)}, 185--232, (2002).
\bibitem{Mcshame}
E. J. McShane:
\emph{Stochastic Calculus and Stochastic Models.}
Academic Press, (1974).
\bibitem{Mei-Yin2015}
H. Mei, G. Yin:
Convergence and convergence rates for approximating ergodic means of functions of solutions to stochastic differential equations with Markov switching.
\emph{Stochastic Process. Appl.,} {\bf 125(8)}, 3104--3125, (2015).
\bibitem{Mohammed}
S. E. A. Mohammed:
\emph{Stochastic Functional Differential Equations.} Pitman Advanced Publishing  Program, (1984).
\bibitem{Monk2003}
N. A. Monk:
Oscillatory expression of Hes1, p53, and NF--$\kappa$B driven by transcriptional time delays.
{Curr. Biol.,} {\bf 13(16)}, 1409--1413, (2003).
\bibitem{Nicaise-Pignotti}
S. Nicaise, C. Pignotti:
Stability and instability results of the wave equation with a delay term in the boundary or internal feedbacks.
{SIAM J. Control Optim.,} {\bf45(5)}, 1561--1585, (2006).


\bibitem{NNN}
D. H. Nguyen, D. Nguyen, S. L. Nguyen: Stability in distribution of path-dependent hybrid diffusion.
{SIAM J. Control Optim., }  {\bf 59(1)}, 434--463, (2021).
\bibitem{Pages-Panloup2012}
G. Pag\`{e}s, F. Panloup:
Ergodic approximation of the distribution of a stationary diffusion: rate of convergence.
{Ann. Appl. Probab.,} {\bf22(3)}, 1059--1100, (2012).
\bibitem{Shaikhet}
L. Shaikhet:
{Lyapunov Functionals and Stability of Stochastic Functional Differential Equations.}
Springer, (2013).
\bibitem{Song-Li2021}
G. Song, J. Hu, S. Gao, X. Li:
The strong convergence and stability of explicit
approximations for nonlinear stochastic delay
differential equations.
{Numer. Algorithms,}  {\bf 89(2)}, 855--883, (2022).
\bibitem{TDHS06}
Y. Takeuchi, N. H. Du, N. T. Hieu, K. Sato:
Evolution of
predator-prey systems described by a Lotka-Volterra equation under
random environment. {J. Math. Anal. Appl.,} {\bf 323(2)}, 938--957, (2006).

\bibitem{Villani}
C. Villani:
{Optimal Transport,} Springer-Verlag, (2009).
\bibitem{WangYa}
Y. Wang, F.  Wu, X.  Mao:
Stability in distribution of stochastic functional differential equations.
{ Systems Control Lett.,} {\bf 132}, 104513, (2019).

\bibitem{Xie-Zhang}
L. Xie, X. Zhang: Ergodicity of stochastic differential equations with jumps and singular coefficients.
Ann. Inst. Henri Poincar\'{e} Probab. Stat., {\bf 56(1)}, 175--229, (2020).
\bibitem{mao-yuan2003}
C. Yuan, X. Mao:
Asymptotic stability in distribution of stochastic
differential equations with Markovian switching,
{Stochastic Process. Appl.,} {\bf 103(2)}, 277--291, (2003).
\bibitem{ZhangX}
X. Zhang: Euler schemes and large deviations for stochastic Volterra equations with singular kernels.
 J. Differential Equations,  {\bf244(9)}, 2226--2250, (2008).
\end{thebibliography}
\end{document}